\documentclass[11pt]{amsart}
\usepackage{amssymb, adjustbox, enumerate, amsbsy}
\usepackage{array, longtable}
\usepackage{geometry}
\usepackage{todonotes}

\usepackage{amsfonts, amssymb, amscd}
\numberwithin{equation}{section}

\usepackage[symbol]{footmisc}

\usepackage{bm}
\usepackage{verbatim}
\usepackage{mathrsfs}
\usepackage{graphicx}
\usepackage{tikz, tikz-cd}
\usetikzlibrary{shapes.geometric}
\usepackage{subcaption}
\usepackage{listings}
\usepackage{subfiles}
\usepackage[toc,page]{appendix}
\usepackage{mathtools}
\usepackage{comment}
\usepackage{enumerate}
\usepackage{enumitem}
\usepackage[all]{xy}

\usepackage{graphicx}
\graphicspath{{images/}}

\usepackage{appendix}
\usepackage{hyperref}
\hypersetup{
    colorlinks=true,
    citecolor=red,
    linkcolor=blue,
    filecolor=magenta,      
    urlcolor=red,
}
\newcommand{\Dd}{{\mathcal{D}}}

\newcommand{\KK}{\mathbb{K}}

\newcommand{\Qq}{\mathbb{Q}}

\newcommand{\QQ}{\mathbb{Q}}
\newcommand{\Rr}{\mathbb{R}}
\newcommand{\RR}{\mathbb{R}}

\newcommand{\ZZ}{\mathbb{Z}}

\newcommand{\nt}{\operatorname{nt}}

\newcommand{\vol}{\operatorname{vol}}

\newcommand{\Exc}{\operatorname{Exc}}

\newcommand{\Aut}{\operatorname{Aut}}

\newcommand{\Nklt}{\operatorname{Nklt}}
\newcommand{\mld}{\operatorname{mld}}

\newcommand{\pet}{\operatorname{pet}}

\newcommand{\Supp}{\operatorname{Supp}}

\newcommand{\Diff}{\operatorname{Diff}}

\newcommand{\mult}{\operatorname{mult}}

\newtheorem{thm}{Theorem}[section]

\newtheorem{cor}[thm]{Corollary}
\newtheorem{lem}[thm]{Lemma}
\newtheorem{prop}[thm]{Proposition}

\newtheorem{claim}[thm]{Claim}

\theoremstyle{definition}
\newtheorem{defn}[thm]{Definition}
\newtheorem{ques}[thm]{Question}
\theoremstyle{definition}
\newtheorem{rem}[thm]{Remark}
\newtheorem{rmk}[thm]{Remark}

\newtheorem{ex}[thm]{Example}

\theoremstyle{definition}

\newcommand{\Gal}{\mathrm{Gal}}
\newcommand{\GL}{\mathrm{GL}}
\newcommand{\Hess}{\mathrm{Hess}}
\newcommand{\I}{\mathrm{I}}
\newcommand{\II}{\mathrm{II}}
\newcommand{\III}{\mathrm{III}}
\newcommand{\klt}{\mathrm{klt}}
\newcommand{\nklt}{\mathrm{nklt}}
\newcommand{\ord}{\mathrm{ord}}
\newcommand{\Proj}{\mathrm{Proj}}

\newcommand{\CC}{\mathbb{C}}
\newcommand{\FF}{\mathbb{F}}

\newcommand{\PP}{\mathbb{P}}

\newcommand{\sD}{\mathcal{D}}
\newcommand{\sO}{\mathcal{O}}
\newcommand{\sS}{\mathcal{S}}

\newcommand{\tS}{\widetilde{S}}

\tikzstyle{wbullet}=[circle, draw=black, fill=white, thick, inner sep=2pt, minimum size=1.5mm]
\tikzstyle{bbullet}=[circle, draw=black, fill=black, inner sep=2pt, minimum size=1.5mm]
\tikzstyle{rbullet}=[circle, draw=black, fill=red, thick, inner sep=2pt, minimum size=1.5mm]
\tikzset{
  symbol/.style={
    draw=none,
    every to/.append style={
      edge node={node [sloped, allow upside down, auto=false]{$#1$}}}
  }
}

\begin{document}

\title[The minimal volume of surfaces with non-empty non-klt locus]{The minimal volume of surfaces of log general type with non-empty non-klt locus}

\author{Jihao Liu and Wenfei Liu}

\address{Department of Mathematics, Northwestern University, 2033 Sheridan Rd, Evanston, IL 60208, USA}
\email{jliu@northwestern.edu}

\address{School of Mathematical Sciences, Xiamen University, Siming South Road 422, Xiamen, Fujian 361005, P. R. China}
\email{wliu@xmu.edu.cn}

\subjclass[2020]{14J29, 14B05, 14E30}
\keywords{stable surface, minimal volume, accumultation point, moduli space}
\date{\today}

\begin{abstract}
We show that the minimal volume of surfaces of log general type, with non-empty non-klt locus on the ample model, is $\frac{1}{825}$. Furthermore, the ample model $V$ achieving the minimal volume is determined uniquely up to isomorphism. The canonical embedding presents $V$ as a degree $86$ hypersurface of $\mathbb P(6,11,25,43)$. This motivates a one-parameter deformation of $V$ to klt stable surfaces within the weighted projective space. Consequently, we identify a \emph{complete} rational curve in the corresponding moduli space.

As an important application, we deduce that the smallest accumulation point of the set of volumes for projective log canonical surfaces equals $\frac{1}{825}$.
\end{abstract}

\maketitle
\tableofcontents

\section{Introduction}\label{sec: introduction}
The volume of a projective surface $X$ with log canonical (lc) singularities measures the growth rate of the pluri-canonical systems:
\[
h^0(X, mK_X) =\frac{\vol(K_X)}{2}m^2 + o(m^2),
\]
and it is a key invariant in the classification theory of both smooth and non-smooth surfaces. If $\vol(K_X)>0$, then we say that $K_X$ is \emph{big}, and $X$ is of \emph{log general type}. The term ``log" indicates that $X$ might have non-canonical singularities, and the terminology generalizes easily to the case of pairs. By the existence of ample models for surfaces of log general type, there is no loss of generality to assume that $K_X$ is ample instead of big. In other words, we may assume that $X$ is a \emph{stable surface}.

Let $\sS$ be the set of projective log canonical surfaces with ample $K_X$. It is a fundamental result of Alexeev \cite{Ale94} that the volume set
$$\KK^2:=\{K_X^2\mid X\in \sS\}$$
satisfies the descending chain condition (DCC)\footnote{This result was expanded to all dimensions in \cite{HMX14}.}. Equivalently, any non-empty subset of $\KK^2$ attains the minimum. A more general version of Alexeev's result includes surfaces with boundary $\Rr$-divisors whose coefficients come from a fixed DCC set. In this paper, we will encounter log canonical surface pairs $(X,B)$ where $B$ is a reduced divisor.

For the sake of classification, it is an interesting problem
to determine the minimal volumes of various naturally appearing classes of stable surfaces. The second author showed that the smallest volume of normal stable surfaces with positive geometric genus $h^0(X,K_X)$ is $\frac{1}{143}$ (see \cite{Liu25}). However, by existing examples, the absolute minimum of $\KK^2$ is attained by a surface with vanishing geometric genus. So far, the smallest known volume is $\frac{1}{48983}$, found by Alexeev and the second author \cite{AL19a}, and later in \cite{Tot23} using a different method. On the other hand, the current proved lower bound for $\KK^2$ is:
\[
\min\KK^2\geq \frac{1}{42\cdot 84^{128 \cdot 42^5+168}}.
\]
This value is much smaller than the conjectured lower bound $\frac{1}{48983}$.

Another feature of $\KK^2$ is that it has accumulation points, which appear as the volumes of surfaces with accessible non-klt centers (see \cite{AL19b}). Given this, it is beneficial to partition $\sS$ into two subsets based on the nature of the singularities. Specifically, we can write
\begin{equation}\label{eq: decompose S}
    \sS= \sS_{\klt}\sqcup \sS_{\nklt},
\end{equation}
where
\[
  \sS_{\klt}:= \{X\in \sS\mid X \text{ is klt}\}
 \text{ and }
 \sS_{\nklt}:= \{X\in \sS\mid X \text{ is not klt} \}.
\]
The respective volume sets are then denoted by
\[
\KK^2_{\klt}:=\{K_X^2\mid X\in  \sS_{\klt} \}  \text{ and }  \KK^2_{\nklt}:=\{K_X^2\mid X\in  \sS_{\nklt} \}.
\]
A primary objective of this paper is to determine the minimum value in $\KK_\nklt^2$:
\begin{thm}\label{thm: 825}
The value $\min\KK_\nklt^2$ is $\frac{1}{825}$, and the surface that achieves this minimum is uniquely determined up to isomorphism, as in Example \ref{ex: lc surface smallest volume}.
\end{thm}

The proof for the lower bound $\min\KK_\nklt^2 \geq\frac{1}{825}$ draws on the recent results of \cite{LS23} and the techniques used therein. The characterization of the surface achieving the minimum, however, relies on the construction from \cite{AL19a} and a detailed classification of K3 surfaces with a non-symplectic automorphism of order $11$ \cite{OZ00, AST11, OZ11}.

The surface discussed in Example \ref{ex: lc surface smallest volume} has an accessible non-klt center, as defined in \cite[Definition 2.3]{AL19b}. This implies that its volume $\frac{1}{825}$ is an accumulation point of $\KK^2$. On the other hand, every accumulation point of $\KK^2$ is realized as the volume of a normal stable surface with non-empty non-klt locus, as stated in \cite[Theorem 1.1]{AL19b}. Consequently, we deduce the following:
\begin{cor}\label{cor: min acc}
The smallest accumulation point of $\KK^2$ is $\frac{1}{825}$.
\end{cor}
Previously, it was only known that the minimal accumulation point of $\KK^2$ lies between $\frac{1}{86436}$ and $\frac{1}{462}$ \cite[Theorem 1.7(iii)]{AL19b}.

\medskip

Due to Theorem \ref{thm: 825} and Corollary \ref{cor: min acc}, it becomes interesting to study the moduli space $M_{\frac{1}{825}}$ of stable surfaces with $K_X^2=\frac{1}{825}$. While Theorem~\ref{thm: 825} assures the uniqueness of the non-klt stable surface with $K_X^2=\frac{1}{825}$, there do exist klt stable surfaces with the same volume: Following the construction of the unique non-klt stable surface $V$ with minimal possible volume in Example \ref{ex: 825}, we manage to embed $V$ into the weighted projective space $\PP(6,11,25,43)$ as a hypersurface of degree $86$, and discover that $V$ deforms to klt hypersurfaces of the same degree. As a result, we identify a complete rational curve in $M_{\frac{1}{825}}$ that is homeomorphic to $\PP^1$. We refer the reader to Theorem~\ref{thm: WPS86} for details.

This phenomenon prompts the following question:
\begin{ques}
   Are there klt stable surfaces with volume $\frac{1}{825}$ that are not isomorphic to a hypersurface of $\PP(6,11,25,43)$?
\end{ques}

The paper is organized as follows. In Section \ref{sec: preliminaries}, we detail the basic notions and facts about the birational geometry and singularities of surfaces. We construct a useful partial resolution of a surface with non-klt singularities and control the order of a point on a plt surface pair under a $K$-positive contraction. Regarding the global geometry of projective surfaces, we demonstrate how to utilize pseudo-effective and nef thresholds to estimate the canonical volume. Moreover, we investigate a special configuration of $(-2)$-curves on a special K3 surface with a non-symplectic automorphism of order $11$ (\cite{OZ11, AST11}). This configuration proves important to the pivotal example, Example~\ref{ex: 825}. In Section \ref{sec: 462}, we build upon the results of \cite{LS23}, proving the uniqueness of the projective log canonical surface pair $(X,B)$ such that $K_X+B$ is ample, $B=\lfloor B\rfloor\neq 0$, and $(K_X+B)^2=\frac{1}{462}$. It is worth noting that this surface is closely connected to Example \ref{ex: 825}. Finally, in Section \ref{sec: 825}, we first construct the stable surface where $\Nklt(X)\neq \emptyset$ and $K_X^2=\frac{1}{825}$, and study its canonical embedding into $\PP(6,11,25,43)$. Then we present the proof of the main theorem, Theorem~\ref{thm: 825}, and provide several variants and applications.

\medskip

\noindent\textbf{Acknowledgement}. Part of this project began when the first author visited Xiamen University in April and June 2023. The first author wishes to extend his appreciation for the warm hospitality he received during his stay. He would also like to thank Vyacheslav V. Shokurov for invaluable discussions related to the subject matter. The second author is grateful for the enlightening conversations and collaboration with Valery Alexeev on the topic of singular surfaces with small volumes. He also extends his gratitude to De-Qi Zhang, who pointed out the reference \cite{OZ11} during his visit to the National University of Singapore in 2019. The first author was supported by the National Key R\&D Program of China (\#2024YFA1014400). The second author was supported by the NSFC (No.~11971399 and 12571046) and the Presidential Research Fund of Xiamen University (No.~20720210006).

\section{Preliminaries}\label{sec: preliminaries}

We work over the field of complex numbers $\mathbb C$. 

\subsection{Birational geometry and singularities of surfaces}
In this subsection, we revisit basic notions and facts related to birational geometry and the singularities of surfaces. Although many concepts are valid across all dimensions \cite{Sho92,KM98,BCHM10}, we confine our discussion to the case of surfaces for simplicity.

A \emph{surface} is defined as a connected quasi-projective variety of dimension $2$. Unless stated otherwise, it is always assumed to be normal. For a field $\FF$, which can be $\QQ$, or $\RR$, an \emph{$\FF$-divisor} $B$ on a surface is a finite formal sum $B=\sum_i b_iB_i$ of prime divisors with $\FF$-coefficients. We say that $B$ is \emph{effective} if $b_i\geq 0$ for every index $i$, and denote $B\geq 0$. The \emph{round-down} of $B$ is given by $\lfloor B\rfloor:=\sum_i \lfloor b_i\rfloor B_i$, with $\lfloor b_i\rfloor$ being the largest integer less than or equal to $b_i$. A surface pair $(X,B)$ consists of a surface $X$ and an $\RR$-divisor $B\geq 0$ on $X$ such that $K_X+B$ is $\RR$-Cartier, and we say that $B$ is the \emph{boundary} of $(X,B)$. A surface $X$ with an $\RR$-Cartier $K_X$ can also be interpreted as a surface pair $(X,0)$. A \emph{germ} of a surface pair $(X\ni x,B)$ consists of a surface pair $(X,B)$ and a closed point $x\in X$.

When selecting the canonical divisors of the surfaces under study, we ensure that for any birational morphism $f\colon Y\rightarrow X$, the relation $f_*K_Y=K_X$ holds. In particular, for any $\Rr$-divisor $B$ on $X$ such that $K_X+B$ is $\RR$-Cartier, there exists a  uniquely determined $\RR$-divisor $B_Y$ on $Y$ such that $K_Y+B_Y:=f^*(K_X+B)$ and $f_*B_Y=B$. 

Let $Y$ be a surface and $D$ an $\RR$-Cartier $\RR$-divisor on $Y$. A birational morphism  $f\colon Y\rightarrow X$ is called \emph{$D$-positive} (resp.~\emph{$D$-non-negative}, resp.~\emph{$D$-negative}) if $D\cdot C$ is positive  (resp.~non-negative, resp. negative) for any curve $C\subset Y$ contracted by $f$.

A birational morphism $f\colon Y\rightarrow X$ between surfaces is called a \emph{resolution} of $X$ if $Y$ is smooth. $f$ is called the \emph{minimal resolution} of $X$ if the exceptional locus of $f$ does not contain any $(-1)$-curves. It is well-known that any resolution of $X$ factors through the minimal resolution of $X$.

For a surface pair $(X,B)$, we call a birational morphism $f\colon Y\rightarrow X$ a \emph{log resolution} of $(X,B)$ if $\Exc(f)\cup f^{-1}_* B$ is a simple normal crossing, where $\Exc(f)$ denotes the exceptional locus of $f$.

Let $(X,B)$ be a surface pair. For a birational morphism $f\colon Y\rightarrow X$ and any prime divisor $E\subset Y$, we say that $E$ is \emph{over $X$}. We say that $c_X(E):=f(E)$ is the \emph{center} of $E$ on $X$, and say that $E$ is \emph{over} $c_X(E)$. The coefficient of $E$ in $B_Y$ is denoted by $\mult_EB_Y$, and $a(E,X,B):=1-\mult_EB_Y$ is called the \emph{log discrepancy} of $E$ with respect to $(X,B)$.  If $a(E,X,B)=0$, then $E$ is called an \emph{lc place} of $(X,B)$, and $c_X(E)$ is called an \emph{lc center} of $(X,B)$. The \emph{non-klt locus} $\Nklt(X,B)$ is the union of all lc centers of $(X,B)$.

Let $(X\ni x,B)$ be a surface germ. 
$$\mld(X\ni x,B):=\inf\{a(E,X,B)\mid E \text{ is a divisor over $x$}\}$$
is called the \emph{mld} of $(X, B)$ at $x$. We say that $(X\ni x, B)$ is 
\[
\begin{cases}
log\,\, canonical\,\, (lc) & \text{if $\mld(X\ni x, B)\geq 0$}\\
Kawamata\,\, log\,\, terminal\,\, (klt) & \text{if $\mld(X\ni x, B)> 0$ and $\lfloor B\rfloor$ does not pass through $x$}\\
purely\,\, log\,\, terminal\,\,(plt) & \text{if $\mld(X\ni x, B)> 0$}.
\end{cases}
\]
It is clear that 
\[
klt \Rightarrow plt \Rightarrow lc.
\]
By \cite[Section 4]{KM98}, if $(X\ni x,B)$ is plt, then $X\ni x$ is klt. If $(X\ni x,B)$ is lc and $x\in \Supp(B)$, then $X\ni x$ is klt. We say that $X\ni x$ is \emph{non-klt} if it is not klt. We say $(X, B)$ is lc (resp.~klt, resp.~plt) if $(X\ni x, B)$ is so for every closed point $x\in X$.

A surface pair $(X,B)$ is called \emph{dlt} if there exists a log resolution $f: Y\rightarrow X$ of $(X,B)$ such that $a(E,X,B)>0$ for any prime $f$-exceptional divisor $E$ on $Y$. For any lc surface pair $(X,B)$, a \emph{dlt modification} of $(X,B)$ is a birational morphism $f: Y\rightarrow X$ which only extracts divisors with log discrepancy $0$ with respect to $(X,B)$, and $(Y,B_Y)$ is dlt, where $K_Y+B_Y:=f^*(K_X+B)$. We recall that any lc surface pair has a dlt modification (cf.~\cite[Proposition 3.3.1]{HMX14}).

To classify lc surface singularities $(X\ni x, B)$, one needs to take the minimal resolution $f\colon Y\rightarrow X$ and looks at its (extended) dual graph, defined as follows. 

\begin{defn}[Dual graph]\label{defn: dual graph}
Let $C=\cup_{i=1}^nC_i$ a collection of irreducible projective curves on a smooth surface. We define the \emph{dual graph} $\Dd(C)$ of $C$ as follows.
\begin{enumerate}
    \item The vertices of $\Dd(C)$ are the curves $C_i$, which is usually decorated with the integer $-C_i^2$, or some other relevant data about the curve $C_i$ if necessary.
    \item For $i\neq j$, the vertices $C_i$ and $C_j$ are connected by $C_i\cdot C_j$ edges.
\end{enumerate}
For convenience, we allow $C$ to be an empty curve, and define $\Dd(\emptyset)=\emptyset$. 

We say that $\mathcal{D}(C)$ contains a \emph{cycle} if there exists $C_{1},\dots,C_{n}$ and edges $e_1, \dots, e_n$ ($n\geq 1$) such that $e_i$ connects $C_i$ and $C_{i+1}$, where $C_{n+1}$ is meant to be $C_1$. A connected graph $\mathcal{D}(C)$ is a \emph{tree} if it does not contain a cycle.

Suppose that $\Dd(C)$ is a tree. A vertex $C_i\in \Dd(C)$ is called a \emph{fork} of $\mathcal{D}(C)$ if $C_i\cdot C_j\geq 1$ for at least three different $j\not=i$. The tree $\Dd(C)$ is called a \emph{chain} if it does not have a fork. A vertex $C_i\in \Dd(C)$ is called a \emph{tail} if $C_i$ intersects $C_j$ for only one $j\not=i$. 

If $C\not=\emptyset$, we call the matrix $(C_i\cdot C_j)_{1\leq i,j\leq n}$ the \emph{intersection matrix} of $\mathcal{D}(C)$. The \emph{determinant} of $\mathcal{D}(C)$ is defined as
\[
\det(\mathcal{D}(C)):=
\begin{cases}
\det(-(C_i\cdot C_j)_{1\leq i,j\leq n}) & \text{if $C\not=\emptyset$} \\
1 & \text{if $C=\emptyset$}.
\end{cases}
\]

Let $(X\ni x,B)$ be a surface germ and $f: Y\rightarrow X$ the minimal resolution of $X$ near $x$. Then we call $\Dd(X\ni x):=\Dd(\Exc(f))$ the dual graph of $X\ni x$, and $\Dd(\Exc(f)\cup f^{-1}_*\Supp B)$ the \emph{extended dual graph} of $(X\ni x,B)$.
\end{defn}

For the purpose of this paper, we divide lc surface singularities into klt singularities and non-klt singularities. 
 We have the following important classification result (cf.~\cite[Theorems~4.7 and 4.15]{KM98} and \cite[3.39]{Kol13}):
\begin{enumerate}
\item A surface singularity $X\ni x$ is klt if and only if it is a quotient singularity, that is, analytically locally $(X\ni x) = (\CC^2\ni 0)/G$, where $G$ is a finite group acting linearly on $\CC^2$ and only the origin $0$ has a possibly non-trivial stabilizer (\cite[Theorem~4.18]{KM98}). In this case, the dual graph $\Dd(X\ni x)$ is a tree with at most one fork, and each vertex corrsponds to a smooth rational curve. Moreover, we have $|G|=\det\Dd(X\ni x)$, which is called the \emph{order} of $X\ni x$ and is denoted by $\ord(X\ni x)$. 
\item An lc surface singularity $(X\ni x, B)$ with $B=\lfloor B\rfloor$ is non-klt if either $B\neq 0$ and $\Dd(X\ni x, B)$ is as in \cite[Theorem~4.15]{KM98}, or $B=0$ and one of the following cases holds:
\begin{enumerate}
    \item $\sD(X\ni x)$ is a smooth elliptic curve or a rational curve with exactly one node. 
    \item $\sD(X\ni x)$ is a cycle of smooth rational curves $E_1,\dots,E_n$:
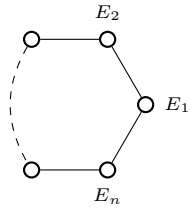
\begin{center}
    \begin{tikzpicture}[font=\tiny]
        \node[wbullet, label=right: $E_1$](E1)at (0:1){};
        \node[wbullet, label=above: $E_2$](E2)at (60:1){};
        \node[wbullet](E3)at (120:1){};
        \node[wbullet](En-1)at (240:1){};
        \node[wbullet,  label=below: $E_n$](En)at (300:1){};
        \draw (E1)--(E2)--(E3);
        \draw (E3)edge[dashed, bend right](En-1);
        \draw (En-1)--(En)--(E1);
    \end{tikzpicture}
\end{center}
such that $E_i^2\leq -2$ for any $1\leq i\leq n$ and there exists at least one $i$ satisfying that $E_i^2\leq -3$.
    \item $\sD(X\ni x)$ is of the following form
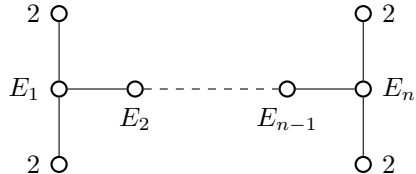
\begin{center}
\begin{tikzpicture}[font=\small]
    \node[wbullet, label=left: $E_1$](E1)at (-2, 0){};
\node[wbullet, label=below: $E_2$](E2)at (-1, 0){};
\node[wbullet, label=below: $E_{n-1}$](E3) at (1,0){};
\node[wbullet, label=right: $E_{n}$](E4)at (2,0){};
   \foreach \y in {-1,1} 
   \node[wbullet, label=left:$2$](A\y) at (-2, \y){};
   \foreach \y in {-1,1} 
   \node[wbullet, label=right:$2$](B\y) at (2, \y){};
   \draw (A-1)--(E1)--(A1);
      \draw (B-1)--(E4)--(B1);
      \draw (E1)--(E2);
      \draw[dashed](E2)--(E3);
      \draw (E3)--(E4);
\end{tikzpicture}
\end{center}
where $E_i^2\leq - 2$ for each $1\leq i\leq n$, and there exists at least one $i$ such that $E_i^2\leq -3$.
\item $\sD(X\ni x)$ is a tree with a unique fork $E_0$ and three branches $\Gamma_i$, $1\leq i\leq 3$:
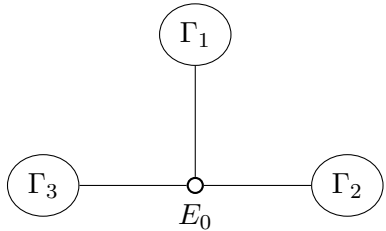
\[
\begin{tikzpicture}
    \node[wbullet, label=below: $E_0$] (E0) at (0,0){};
    \node[draw, shape=ellipse] (Ga1) at (0,2){$\Gamma_1$};
        \node[draw, shape=ellipse] (Ga3) at (2, 0){$\Gamma_2$};
            \node[draw, shape=ellipse](Ga2) at (-2, 0){$\Gamma_3$};
    \draw (E0)--(Ga1);
    \draw (E0)--(Ga2);
    \draw (E0)--(Ga3);
\end{tikzpicture}
\]
\end{enumerate}
where, up to relabelling,
\[
\left(\det(\Gamma_1),\det(\Gamma_2), \det(\Gamma_3)\right)\in \{(2,3,6), \,\, (2,4,4),\,\, (3,3,3)\}.
\]
\end{enumerate}
We read from the above classification that, if $(X\ni x, B)$ with $B=\lfloor B\rfloor\neq 0$ passing through $x$ is plt, then $x$ is a cyclic quotient singularity of $X$ and a non-singular point of $B$, and $f^{-1}_*B$ is a tail of $\Dd(X\ni x, B)$.

For an lc surface singularity $X\ni x$, the divisors $E$ over $x$ which compute $\mld(X\ni x)$ play an important role. If $X\ni x$ is non-klt, then such divisors are exactly the lc places of $X$ over $x$. Usually, we extract these divisors using a birational morphism, aiming to enhance the singularities and thereby gaining flexibility to adjust the boundary divisor.
\begin{lem}\label{lem: e^2 and -1/3}
    Let $X\ni x$ be an lc surface germ that is not klt. Then there exists a birational morphism $f\colon Y\rightarrow X$ satisfying the following. Let $E$ be the reduced $f$-exceptional divisor. Then the following statements hold.
    \begin{enumerate}
    \item Any component of $E$ is an lc place of $X$.
    \item The center of any component of $E$ on $X$ is $x$.
    \item $Y$ is klt near $E$.
    \item $-E$ is ample$/X$. 
    \item $E^2\leq-\frac{1}{3}$, and if $\mathcal{D}(X\ni x)$ is not of the following two types in Table \ref{table: two special singularity}, then $E^2\leq -\frac{1}{2}$.
        \begin{center}
    \begin{longtable}{|c|c|}
 			\caption{The graphs of two special singularities with $\mld=0$}\label{table: two special singularity}\\
		\hline
		No.& $\mathcal{D}(X\ni x)$ \\ 
			\hline
   1 &
 \begin{tikzpicture}[font=\tiny, scale=.75]
 \node[wbullet, label=below:$2$] (E0) at (0,0){};
 \node[wbullet, label=below:$2$] (E1) at (-2,0){};
\node[wbullet, label=below:$2$] (E2) at (2,0){};
\node[wbullet, label=right:$3$] (E3) at (0,1){};
\draw (E0)--(E1)node[wbullet, pos=1/2, label=below:$2$](E01){};
\draw (E0)--(E2)node[wbullet, pos=1/2, label=below:$2$](E02){};
\draw (E0)--(E3);
\end{tikzpicture}
\\
    	\hline	
		   2 &
     \begin{tikzpicture}[font=\tiny, scale=.75]
 \node[wbullet, label=below:$2$] (E0) at (0,0){};
 \node[wbullet, label=below:$2$] (E1) at (-5,0){};
\node[wbullet, label=right:$2$] (E2) at (0,1){};
\node[wbullet, label=below:$3$] (E3) at (1,0){};
\draw (E0)--(E1)node[wbullet, pos=1/5, label=below:$2$](E01A){}node[wbullet, pos=2/5, label=below:$2$](E01B){}node[wbullet, pos=3/5, label=below:$2$](E01C){}node[wbullet, pos=4/5, label=below:$2$](E01D){};
\draw (E0)--(E2);
\draw (E0)--(E3);
\end{tikzpicture}
\\
    	\hline
		\end{longtable}
		\end{center}
    \item If $\mathcal{D}(X\ni x)$ is of one of the two types in Table \ref{table: two special singularity}, then $E$ is the image of the unique fork of $\mathcal{D}(X\ni x)$ and $E^2=-\frac{1}{3}$.

    \end{enumerate}
\end{lem}
\begin{proof}
Let $h\colon \widetilde X\rightarrow X$ be the minimal resolution of $X$ near $x$, and write $K_{\widetilde X}+B_{\widetilde X}=h^*K_X$. Let $E_i$ be the $h$-exceptional prime divisors and write $B_{\widetilde X}=\sum_i b_iE_i$. Let $g\colon {\widetilde X}\rightarrow Y$ be the contraction of all the $E_i$ with coefficient $b_i<1$ or with $K_{\widetilde X}\cdot E_i=0$, and let $f\colon Y\rightarrow X$ be the induced proper birational morphism. In fact, $g\colon {\widetilde X}\rightarrow Y$ simply gives the the ample model of $K_{\widetilde X}+B_{\widetilde X}'$ over $X$, where $B_{\widetilde X}'$ is the reduced divisor consisting of those components $E_i$ contracted by $g$. 

Setting $E=g_*B_{\widetilde X}$, it is quite clear that the properties (1)--(4) hold. For the properties (5) and (6), we use the classification of lc but not klt surface singularities given above. In the cases (a), (b), (c) of the classification, $g$ contracts exactly the $(-2)$-curves over $x$, and we have $B_{\widetilde X} = g^*E$; it follows that $E^2 = B_{\widetilde X}^2$, which is a negative integer and hence at most $-1$. Thus in these cases the property (5) holds and (6) is an empty statement.

In the case (d) of the classification, $\sD(X\ni x)$ is a tree with a unique fork $E_0$ and three branches $\Gamma_i = \cup_{j=1}^{r_i}E_{i,j}$, $1\leq i\leq 3$, where the indices are ordered so that for each $1\leq i\leq 3$,
\begin{itemize}
    \item $E_{i,j}$ intersects $E_{i',k}$ if and only if $|j-k|\leq 1$,
    \item $E_{i,1}$ intersects $E_0$.
\end{itemize}
We may write $\Gamma_i = E_{1,1}\cup \Gamma_i'$, where $\Gamma_i' = \cup_{j=2}^{r_i}E_{i,j}$ is a possibly empty subgraph of $\Gamma_i$. Then the dual graph $\sD(X\ni x)$ takes the following form
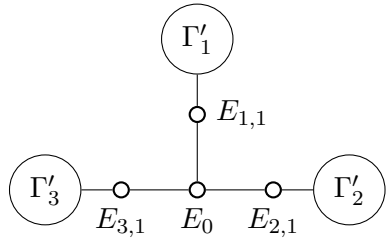
\[
\begin{tikzpicture}
    \node[wbullet, label=below: $E_0$] (E0) at (0,0){};
     \node[wbullet, label=right: $E_{1,1}$] (E1) at (0,1){};
    \node[draw, shape=ellipse] (Ga1) at (0,2){$\Gamma_1'$};
        \node[wbullet, label=below: $E_{2,1}$] (E2) at (1, 0){};
        \node[draw, shape=ellipse] (Ga2) at (2, 0){$\Gamma_2'$};
    \node[wbullet, label=below: $E_{3,1}$] (E3) at (-1, 0){};
    \node[draw, shape=ellipse](Ga3) at (-2, 0){$\Gamma_3'$};
    \draw (E0)--(E1)--(Ga1);
    \draw (E0)--(E2)--(Ga2);
    \draw (E0)--(E3)--(Ga3);
\end{tikzpicture}
\]
Let $n_i:=\det \Gamma_i$ and  
$q_i:=\det \Gamma_i'$. 
Then after reordering indices, we have
\[
(n_1,n_2,n_3)\in\{(3,3,3),(2,4,4),(2,3,6)\}.
\]
The three branches $\Gamma_i$ are contracted by $g$, and we have (cf.~\cite[Notation--Lemma~4.5]{LS23})
\[
E^2=E_0^2+\sum_{i=1}^3\frac{q_i}{n_i}.
\]
Since $(n_1,n_2,n_3)\in\{(3,3,3),(2,4,4),(2,3,6)\}$, and $q_i$ and $n_i$ are coprime, either $2E^2$ or $3E^2$ is an integer. Since $E^2<0$, we have $E^2\leq-\frac{1}{3}$. Moreover, it is easy to check that $E^2=-\frac{1}{3}$ if and only if $E_0^2=-2$, and possibly reordering indices,
$$(n_1,n_2,n_3,q_1,q_2,q_3)\in \{(3,3,3,1,2,2),(2,3,6,1,1,5)\}.$$
This implies the required properties (5) and (6) in this case.
\end{proof}

We control the order of a point on a plt surface pair under a $(K_X+B)$-positive contraction.
\begin{lem}\label{lem: index of positive contraction}
    Let $(X\ni x,B)$ be a germ of an lc surface pair such that $B$ is a reduced curve passing through $x$. Let $f\colon X\rightarrow Y$ be a proper birational morphism that is $(K_X+B)$-positive, $B_Y:=f_*B$, and $y:=f(x)$. Suppose that $(Y,B_Y)$ is plt. Let $p$ be the order of $X\ni x$ and $q$ the order of $Y\ni y$. Then the following holds.
    \begin{enumerate}
       \item $\ord(X\ni x)\leq \ord(Y\ni y)$, and equality holds if and only if $f$ is an isomorphism near $x$.
       \item If $\ord(X\ni x)= \ord(Y\ni y) + 1$, then $f^{-1}(y)$ is a smooth rational curve.
    \end{enumerate}
\end{lem}
\begin{proof}
Since $f$ is $(K_X+B)$-positive, by the negativity lemma (\cite[Lemma~3.39]{KM98}), $f^*(K_Y+B_Y) - (K_X+B)$ is an effective divisor with support $\Exc(f)$. Since $(Y,B_Y)$ is plt, we infer that $(X,B)$ is also plt. By \cite[Lemma 7.5]{LS23}, $$\mld(X\ni x,B)=\frac{1}{\ord(X\ni x)} \text{ and } \mld(Y\ni y,B_Y)=\frac{1}{\ord(Y\ni y)}.$$ We obtain (1) by the effectivity of $f^*(K_Y+B_Y) - (K_X+B)$.

 Since $(Y,B_Y)$ is plt, $f^{-1}(y)$ consists of smooth rational curves in any case. Suppose on the contrary that $f^{-1}(y)$ contains $2$ curves, say $E_1,E_2$. Then we may assume that $E_1$ passes through $x$ and $E_2$ intersects $E_1$. We let $h: X\rightarrow W$ be the contraction of $E_1$, $w:=h(x),B_W:=h_*B$, and let $g: W\rightarrow Y$ be the induced contraction:
    \[
\begin{tikzcd}
    (X\owns x, B)\arrow[rd, "f"']\arrow[rr, "h"] & & (W\owns w, B_W)\arrow[ld, "g"] \\
    & (Y\owns y, B_Y) &
\end{tikzcd}
\]
Since $(Y,B_Y)$ is plt, $(W,B_W)$ is plt. Since $h$ is not an isomorphism near $x$ and $g$ is not an isomorphism near $w$, by (1), we have
    $\ord(X\ni x)< \ord(W\ni w) <\ord(Y\ni y)$
    so $\ord(Y\ni y)\geq \ord(X\ni x)+2$, which is a contradiction to the assumption.
\end{proof}

\subsection{Volumes}
Now we turn to the global geometry of projective surfaces.

\begin{defn}[Volume]
 Let $X$ be a projective surface, and $D$ an effective $\RR$-Cartier $\Rr$-divisor on $X$. Then \emph{volume} of $D$ is defined as
\[
\vol(D):={\limsup}_{m\rightarrow+\infty}\frac{h^0(X,\lfloor mD\rfloor)}{m^{2}/2}.
\]
The $\RR$-divisor $D$ is said to be \emph{big} if $\vol(D)>0$. The volume only depends only on the numerical class of $D$ and thus defines a function on the space  $N^1(X)_\RR$ of numerically classes of $\RR$-Cartier $\RR$-divisors; it turns out to be a continuous function. Big $\RR$-divisors form an open cone in $N^1(X)_\RR$, called the \emph{big cone} of $X$. An $\RR$-Cartier $\RR$-divisor $D$ is called \emph{pseudo-effective}, if its numerical class lies in the closure of the big cone; it is called \emph{nef}, if $D\cdot C\geq 0$ for any curve $C\subset X$. Pseudo-effective (resp.~nef) $\RR$-divisors are limits of big (resp.~ample) $\RR$-divisors. By the Riemann--Roch theorem, if $D$ is nef, then $\vol(D)=D^2$. We refer to \cite{Laz04} for further general properties of volumes. 

\end{defn}
We are mainly interested in the volumes of log canonical divisors $K_X+B$ of surface pairs $(X, B)$. For simplicity, we call $\vol(K_X+B)$ the \emph{(canonical) volume} of $(X, B)$. The surface pair $(X, B)$ is called \emph{of log general type} if $\vol(K_X+B)>0$. By the minimal model program for projective surfaces (cf.~\cite{Fuj21}), for any lc surface pair $(X, B)$ of log general type, there is a unique birational morphism $\pi\colon X\rightarrow \bar X$, such that $(\bar X,\bar B = \pi_*B)$ is lc, $K_{\bar X}+\bar B$ is ample, and $K_X+B\geq \pi^*(K_{\bar X}+\bar B)$. $\pi$ is called the \emph{ample model} of $K_X+B$ or of $(X, B)$. It is easy to see that 
\begin{equation}\label{eq: ample model}
    \vol(K_X+B) = \vol(K_{\bar X}+\bar B)=(K_{\bar X}+\bar B)^2 = (\pi^*(K_{\bar X}+\bar B))^2.
\end{equation}
Indeed, $\pi^*(K_{\bar X}+\bar B)$ is the positive part of $K_X+B$ in the Zariski decomposition, and \eqref{eq: ample model} is an important instrument to compute the volumes in practice. Thus, from the point of view of volumes, we may assume that $K_X+B$ is ample instead of big. In other words, we may assume that $(X, B)$ is a \emph{stable surface pair}. In general, stable pairs admit possibly non-normal singularities, but the ones encountered in this paper are automatically normal (see Corollary~\ref{cor: normal}). We refer to \cite{Kol23} for the moduli theory of stable pairs.

To bound the volume of a surface pair of log general type from below, we employ two invariants: the nef threshold and the pseudo-effective threshold.

\begin{defn}\label{defn: pet}
Let $(X,B)$ be a projective lc surface pair and $D\geq 0$ a non-zero $\Rr$-Cartier $\Rr$-divisor on $X$.
\begin{itemize}
    \item If $K_X+B+D$ is pseudo-effective, then the \emph{pseudo-effective threshold} of $D$ with respect to $(X,B)$ is
$$\pet(X,B;D):=\inf\{t\geq 0\mid (X,B+tD)\text{ is lc},  K_X+B+tD\text{ is pseudo-effective}\}.$$
\item If $K_X+B+D$ is nef, then the \emph{nef threshold} of $D$ with respect to $(X,B)$ is
$$\nt(X,B;D):=\inf\{t\geq 0\mid (X,B+tD)\text{ is lc},  K_X+B+tD\text{ is nef}\}.$$
\end{itemize}
\end{defn}
Both $\pet(X,B;D)$ and $\nt(X,B;D)$ are non-negative real numbers. When both quantities are defined, it holds that $\pet(X,B;D)\leq\nt(X,B;D)$. A key property we leverage in this paper is the existence of gaps among the pseudo-effective thresholds. The recent discovery of explicit gaps in these thresholds by the first author and Shokurov proves crucial to our work here.
\begin{thm}\label{thm: gap pet}
Let $(X,A+B)$ be a projective lc surface pair such that $A$ and $B$ are reduced, $B\neq 0$, and $K_X+A+B$ is big and nef. Then $\pet(X, A; B)=\frac{12}{13}$ or $\leq \frac{10}{11}$. Moreover, if $\pet(X, A; B)=\frac{12}{13}$, then $A=0$ and $B$ is a smooth rational curve containing only cyclic quotient singularities of $X$ of order $\leq 5$.
\end{thm}
\begin{proof}
This is essentially \cite[Theorem 7.7]{LS23}. We explain how the dlt condition there can be relaxed. Let $f\colon Y\rightarrow X$ be a dlt modification of $(X,A+B)$, $B^Y:=f^{-1}_*B$, and $K_Y+A_Y+B^Y:=f^*(K_X+A+B)$. Then $A_Y$ and $B^Y$ are reduced Weil divisors, $(Y,A_Y+B^Y)$ is $\Qq$-factorial dlt, $\pet(Y,A_Y;B^Y)=\frac{12}{13}$, $K_Y+A_Y+B^Y$ is nef, and $$(K_Y+A_Y+B^Y)\cdot B^Y=(K_X+A+B)\cdot B>0.$$  By \cite[Theorem 7.7]{LS23}, $A_Y=0$, $B^Y$ is a smooth rational curve, and $B^Y$ only contains cyclic quotient singularities of $Y$ of order $\leq 5$.  Thus $f$ is the identity morphism, and the theorem is proven.
\end{proof}
We illustrate the estimation of the volume using the gaps of pseudo-effective thresholds by a simple computation in the situation of Theorem~\ref{thm: gap pet}: Denote $c=\pet(X, 0; B)$. If $(K_X+B)\cdot B>0$, then it is easy to see that  $(K_X+B)\cdot B>\frac{1}{42}$, and  
\begin{align*}
    (K_X+B)^2 &=(K_X+B)\cdot (K_X+cB+(1-c)B) \\
          &= (K_X+B)\cdot (K_X+cB) + (1-c)(K_X+B)\cdot B \\
          &\geq  (1-c)(K_X+B)\cdot B \\
          & \geq \frac{1}{13}\cdot \frac{1}{42} = \frac{1}{546}.
\end{align*}
This lower bound is already very close to the optimal one $\frac{1}{462}$ obtained in \cite{LS23}; see also Theorem~\ref{thm: LS 462}.

The following proposition improves the lower bound when the pseudo-effective threshold achieves the maximum.
\begin{prop}\label{prop: qfact 12/13 smallest volume}
    Keep the notation and assumptions of Theorem~\ref{thm: gap pet}. Suppose that the pseudo-effective threshold $\pet(X,0;B)=\frac{12}{13}$ and $(K_X+B)\cdot B>0$. Then $(K_X+B)^2\geq\frac{1}{260}$.
\end{prop}
\begin{proof}
By Theorem~\ref{thm: gap pet}, $X$ is klt, and $B$ is a smooth curve which only contains cyclic quotient singularities $x_1, \dots, x_m$ of $X$ of order $\leq 5$. Let $n_i$ be the order of $x_i\in X$ for each $i$, then
    $$0<(K_X+B)\cdot B=-2+\sum_{i=1}^m\left(1-\frac{1}{n_i}\right).$$
    Since $n_i\leq 5$ for each $i$, we have $$-2+\sum_{i=1}^m\left(1-\frac{1}{n_i}\right)\geq\frac{1}{20}.$$ Since $(K_X+B)$ is nef,
    $$(K_X+B)^2=(K_X+B)\cdot\left(K_X+\frac{12}{13}B\right)+\frac{1}{13}(K_X+B)\cdot B\geq\frac{1}{13}(K_X+B)\cdot B\geq\frac{1}{260}.$$
\end{proof}

When $(K_X+B)\cdot B=0$, we need to make a more delicate use of the other invariant, the nef threshold $\nt(X,0;B)$, to obtain the desired optimal lower bound; see the proof of Theorem~\ref{thm: 825} in Section \ref{sec: 825}.

\subsection{A K3 surface with a non-symplectic autmorphism of order 11}
A special $K3$ surface $S$ with non-symplectic autmorphism of order 11, found by \cite{OZ00} and further studied by \cite{OZ11, AST11}, turns out to be relevant to the key Example~\ref{ex: 825} of the current paper. We show in the following lemma the uniqueness of certain configuration of $(-2)$-curves on $S$.
\begin{lem}\label{lem: K3}
    Let $S$ be a smooth K3 surface with a non-symplectic autmorphism $\sigma$  of order 11. Suppose that the fixed locus $S^\sigma$ contains a rational curve $B$. Then, the isomorphism class of $S$ is uniquely determined, and up to automorphisms of $S$, there is a unique configuration of ten $(-2)$-curves containing $B$ of the following form:
    \[
\begin{tikzpicture}
    \node[wbullet, label=below: $B$] (B) at (0,0){};
    \node[wbullet, label=right: $2$] (E1) at (0,1){};
    \node[wbullet, label=below: $2$] (E21) at (1, 0){};
    \node[wbullet, label=below: $2$] (E22) at (2, 0){};
    \node[wbullet, label=below: $2$] (E31) at (-1, 0){};
    \node[wbullet, label=below: $2$] (E32) at (-2, 0){};
    \node[wbullet, label=below: $2$] (E33) at (-3, 0){};
    \node[wbullet, label=below: $2$] (E34) at (-4, 0){};
    \node[wbullet, label=below: $2$] (E35) at (-5, 0){};
    \node[wbullet, label=below: $C$] (E36) at (-6, 0){};
    \draw (B)--(E1);
    \draw (B)--(E21)--(E22);
    \draw (B)--(E31)--(E32)--(E33)--(E34)--(E35)--(E36);
\end{tikzpicture}
\]
\end{lem}

\begin{proof}
By \cite[Theorem 7.3]{AST11} and \cite[page 1668]{OZ11}, we know that the invariant lattice $H^2(S, \ZZ)^{\sigma}\cong U\oplus A_{10}$ is exactly the Picard lattice of $X$. In particular, any $(-2)$-curve is stable under the action of $\sigma$. 

By \cite[Proposition 3.1 and Section 4]{OZ00}, there is a $\sigma$-invariant elliptic fibration $f\colon S\rightarrow \PP^1$ with a section such that 
\begin{enumerate}
    \item the fibers $F_0$ and $F_\infty$ over $0$ and $\infty$ are of Kodaira types $\II^*$ and $\III$ respectively, and the remaining 11 singular fibers are of type $\I$, and
    \item the  Weierstrass model of $f$ with the action of $\sigma$ can be given as follows:
   \[
   y^2 = x^3 + t^7x +t^5,\,\, (x,y,t)\mapsto (\zeta^9x, \zeta^8y, \zeta t)
   \]
   where $\zeta$ is a primitive $11$-th root of unity.
\end{enumerate}
It follows that the isomorphism class of $S$ is uniquely determined. Also, $B$ is neccesarily contained in the fiber $F_0$, and $F_0+C$ gives the required configuration of $(-2)$-curves, where $C$ is a section of $f$. For another choice of section $C'$, there is automorphism $\tau\in \Aut(S)$ preserving each fiber of $f$ such that $\tau(C)=C'$, and hence $\tau(F_0+C) = F_0+C'$. 

Conversely, given any configuration $\widetilde F$ as in the theorem, $F=\widetilde F-C$ is of Kodaira type $\II^*$ and it defines a fibration $f\colon S\rightarrow\PP^1$ such that $F_0=F$ and $C$ is a section of $f$. Thus $\widetilde F = F_0+C$ as above.
\end{proof}
\begin{rmk}
    There are two other $\sigma$-invariant elliptic fibrations on $S$; see \cite[Examples 7.1 and 7.4]{AST11} for their Weierstrass models together with the actions of $\sigma$.
\end{rmk}
\begin{rmk}
 Let $\Delta$ be the reduced divisor on $S$  with support exactly the configuration of ten $(-2)$-curves as in Lemma~\ref{lem: K3}. Let $\pi\colon \tS\rightarrow S$ be the blow-up of the nodes of $\Delta$, and $\widetilde\Delta=\pi^{-1}_*\Delta$ the strict transform of $\Delta$ on $\tS$. Then the log canonical model $Z$ of the pair $(\tS,\widetilde\Delta)$ is a stable surface with geometric genus $h^0(Z,K_Z)=1$ and volume $K_Z^2=\frac{1}{143}$, which is the minimal possible for stable surfaces with positive geometric genus (\cite{Liu25}).
\end{rmk}

\section{The log canonical surface \texorpdfstring{$(U,B_U)$}{} with \texorpdfstring{$B_U=\lfloor B_U\rfloor\neq 0$}{} and \texorpdfstring{$(K_U+B_U)^2=\frac{1}{462}$}{} }\label{sec: 462}
We recall the following result of the first author and Shokurov.
\begin{thm}[{\cite[Theorem 1.4]{LS23}}]\label{thm: LS 462}
For any projective lc surface pair $(X,B)$ with $B$ a non-zero reduced divisor and $K_X+B$ ample, we have 
\begin{equation}\label{eq: 462}
    (K_X+B)^2\geq\frac{1}{462}.
\end{equation}
\end{thm}
The lower bound in \eqref{eq: 462} is attained by an lc surface pair. We recall its construction in the following example.
\begin{ex}[{\cite[Section 5]{AL19a}}]\label{ex: 462}
Let $L_0, L_1, L_2, L_3$ be four lines in general position on $\PP^2$, so their dual graph is 
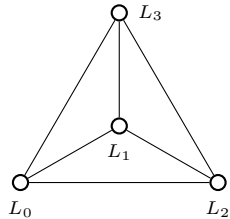
\begin{center}
\begin{tikzpicture}[font=\tiny, scale=1.5]
\node[wbullet, label=below:$L_2$] (L2) at (330:1) {};
\node[wbullet, label=right:$L_3$] (L3) at (90:1) {};
\node[wbullet, label=below:$L_0$] (L0) at (210:1) {};
\node[wbullet, label=below:$L_1$](L1) at (0,0){}
edge (L0)
edge  (L2)
edge  (L3);
\draw (L0)--(L2)--(L3)--(L0);
\end{tikzpicture}
\end{center}
Let $g_U\colon \widetilde U\rightarrow \PP^2$ be the blow-up of some of the nodes of $\sum_{0\leq i\leq 3} L_i$, so that $g_U^{-1}(\sum_{0\leq i\leq 3} L_i)$ has the following dual graph
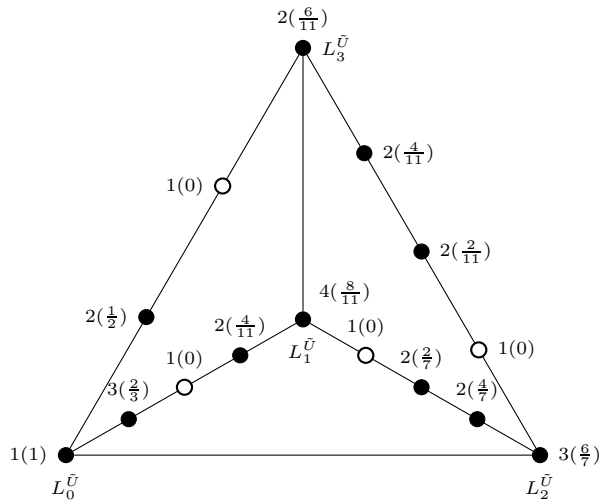
\begin{center}
\begin{tikzpicture}[font=\tiny, scale=3.6]
\node[bbullet, label=below:$L_0^{\tilde U}$, label=left: $1(1)$] (L0) at (210:1) {};
\node[bbullet, label=below:$L_2^{\tilde U}$,  label=right:$3(\frac{6}{7})$] (L2) at (330:1) {};
\node[bbullet, label=below:$L_1^{\tilde U}$,  label=north east:$4(\frac{8}{11})$](L1) at (0,0){};
\node[bbullet, label=right:$L_3^{\tilde U}$, label=above:$2(\frac{6}{11})$] (L3) at (90:1) {};
\draw (L0) -- (L3)node[bbullet, pos=1/3, label=left:$2(\frac{1}{2})$](L03A){}node[wbullet, pos=2/3, label=left:$1(0)$](L03B){};
\draw (L0) -- (L1)
node[bbullet, pos=1/4, label=above:$3(\frac{2}{3})$](L01A){}
node[wbullet, pos=2/4, label=above: $1(0)$](L01B){}
node[bbullet, pos=3/4, label=above:$2(\frac{4}{11})$](L02C){};
\draw (L0) --(L2);
\draw (L2) -- (L1) node[bbullet, pos=1/4, label=above:$2(\frac{4}{7})$](L12A){}node[bbullet, pos=2/4, label=above:$2(\frac{2}{7})$](L12B){}node[wbullet, pos=3/4, label=above:$1(0)$](L12C){};
\draw (L1) -- (L3);
\draw (L2) -- (L3)node[wbullet, pos=1/4, label=right: $1(0)$](L23A){}node[bbullet, pos=2/4, label=right: $2(\frac{2}{11})$](L23B){}node[bbullet, pos=3/4, label=right:$2(\frac{4}{11})$](L23C){};
\end{tikzpicture}
\end{center}
where
\begin{itemize}
    \item $L_i^{\widetilde U}$ denotes the strict transforms of $L_i$ for $0\leq i\leq 3$;
    \item a white bullet denotes the exceptional $(-1)$-curve of the last blow-up over a node;
    \item a black bullet denotes either an $L_i^{\widetilde U}$ or a $(-e)$-curve with $e\geq 2$;
    \item the numbers outside the brackets denote the negative of the self-intersections of the corresponding curves;
    \item the numbers in the brackets denote the coefficients of the curve in the divisor $B_{\widetilde U}$ (to be defined below).
\end{itemize} 
We call the curves appearing on the above dual graph \emph{visible} curves on $\widetilde U$. Let $\widetilde B$ be the reduced divisor consisting of all the curves denoted by the black bullets in the dual graph. Then $(\widetilde X, \widetilde B)$ is of log general type, and one can take its ample model $f_U\colon (\widetilde U, \widetilde B)\rightarrow (U, B_U)$. Then all the curves corresponding to the black bullets except $L_0^{\tilde U}$ are contracted by $f_U$, $B_U=f_{U*}L_0^{\tilde U}$ is irreducible and reduced, $K_{U}+B_U$ is ample, and $(K_{U}+B_U)^2=\frac{1}{462}$. 

On the other hand, $f_U$ can also be viewed as the minimal resolution of $U$. Let $B_{\widetilde U}$ be the $\RR$-divisor on $\widetilde U$ such that $K_{\widetilde U} + B_{\widetilde U} = f_U^*(K_{U}+B_U)$ and $f_{U*}B_{\widetilde U} = B_U$. Then the coefficients of the visible curves in $B_{\widetilde U}$ are the numbers in the brackets of the above dual graph. 

The construction can be illustrated by the following diagram:
\[
\begin{tikzcd}
    (\widetilde U, \widetilde B) \arrow[r, symbol = \supset]\arrow[d, "g_U"']& (\widetilde U, B_{\tilde U})  \arrow[d, "f_U"] \\
    (\PP^2, \sum_{i=0}^{3} L_i) & (U, B_U)
\end{tikzcd}
\]
\end{ex}

\begin{rem}
 In a very recent paper \cite{ET23}, it is shown that a general degree $42$ hypersurface $X\subset\mathbb P(6,11, 14, 21)$ with weighted homogeneous coordinates $x_i$ ($0\leq i\leq 3$), and $B=(x_1=0)$ satisfies that $(X,B)$ is lc, $K_X+B$ is ample, and $\vol(K_X+B)=\frac{1}{462}$. By the following Theorem \ref{thm: 462}, the surface pair $(U,B_U)$ constructed in Example \ref{ex: 462} is isomorphic to the surface pair constructed in \cite{ET23}.
\end{rem}

The goal of the rest part of this section is to complement Theorem~\ref{thm: LS 462} by characterizing the surface achieving the minimal volume. This is a key step in showing the uniqueness part of Theorem \ref{thm: 825}.
\begin{thm}\label{thm: 462}
    Let $(X,B)$ be a projective lc surface pair such that $B$ is a non-zero reduced divisor, $K_X+B$ ample, and $\vol(K_X+B)=\frac{1}{462}$. Then $(X,B)$ is isomorphic to the lc surface pair $(U, B_U)$ constructed in Example~\ref{ex: 462}.
\end{thm}

Before giving the proof of Theorem~\ref{thm: 462}, we control the singularities and the pseudo-effective threshold of $(X, B)$, based on the results of \cite{LS23}.
\begin{lem}\label{lem: basic property log surface min vol}
      Keep the notation and assumptions as in Theorem~\ref{thm: 462}.  Then the following holds.
\begin{enumerate}
        \item $B$ is a smooth rational curve, $(X,B)$ is plt. In particular, $X$ is klt.
        \item $K_X+\frac{10}{11}B\sim_{\mathbb Q}0$.
        \item Let $f: \widetilde X\rightarrow X$ be the minimal resolution of $X$, and $\widetilde B = f^{-1}_*B$. Then the dual graph of $\Supp f^{-1}(B)$ is as follows:
        
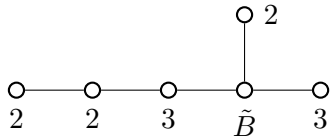
\[
\begin{tikzpicture}
    \node[wbullet, label=below: $\tilde B$] (B) at (0,0){};
    \node[wbullet, label=below: $3$] (E2) at (1, 0){};
    \node[wbullet, label=right: $2$] (E1) at (0,1){};
    \node[wbullet, label=below: $3$] (E3) at (-1, 0){};
    \node[wbullet, label=below: $2$] (E31) at (-2, 0){};
    \node[wbullet, label=below: $2$] (E32) at (-3, 0){};
    \draw (B)--(E1);
    \draw (B)--(E2);
    \draw (B)--(E3)--(E31)--(E32);
\end{tikzpicture}
\]
where $\widetilde B$ is a $(-1)$-curve. In particular, $B$ contains exactly three singular points of $X$, which are all cyclic quotient singularities, and their orders are equal to $2,3,7$ respectively.
        \end{enumerate}
\end{lem}
\begin{proof}
(1) We let $g: Y\rightarrow X$ be a dlt modification of $(X,B)$, and $B^Y:=g^{-1}_*B$. Let $B_Y$ be the effective divisor on $Y$ such that $K_Y+B_Y=g^*(K_X+B)$ and $g_*B_Y=B$. By \cite[Theorem 7.7]{LS23}, there are two cases:

\medskip

\noindent\textbf{Case 1}. $\pet(Y,0;B^Y)=\frac{12}{13}$ and $B^Y$ is a smooth rational curve. In this case, $g$ is the identity morphism, $B$ is a smooth rational curve, and $(X,B)$ is $\Qq$-factorial plt. In particular, $X$ is klt.

\medskip

\noindent\textbf{Case 2}.  $\pet(Y,0;B^Y)\leq\frac{10}{11}$.  In this case, $\pet(Y,B_Y-B^Y;B^Y)\leq\frac{10}{11}$. Since $K_X+B$ is ample, $$(K_Y+B_Y)\cdot B^Y=(K_X+B)\cdot B>0.$$
By \cite[Lemma 7.8]{LS23}, $B_Y=B^Y$ and $B^Y$ is a smooth rational curve. Thus $g$ is the identity morphism, $B$ is a smooth rational curve, and $(X,B)$ is $\Qq$-factorial plt. In particular, $X$ is klt.

(2) Let $c:=\pet(X,0;B)$. Let $x_1,\dots,x_m$ be the singular points of $X$ lying on $B$, and let $n_i$ be  the order of $x_i\in X$ for $1\leq i\leq m$. We may assume that $n_i\leq n_j$ when $i\leq j$. By \cite[Lemma 7.9]{LS23}, $c\leq\frac{10}{11}$. In particular, $K_X+\frac{10}{11}B$ is pseudo-effective. We have
\begin{align*}
    \frac{1}{462}&=(K_X+B)^2=\left(K_X+\frac{10}{11}B\right)\cdot (K_X+B)+\frac{1}{11}(K_X+B)\cdot B\\
    &\geq \frac{1}{11}(K_X+B)\cdot B\\
    &=\frac{1}{11}\cdot\left(-2+\sum_{i=1}^m\left(1-\frac{1}{n_i}\right)\right)\geq\frac{1}{11}\cdot \frac{1}{42}=\frac{1}{462}.
\end{align*}
Therefore, the above inequalities are all equalities. It follow that $m=3$, $(n_1,n_2,n_3)=(2,3,7)$, and  
\begin{equation}\label{eq: 10/11B}
\left(K_X+\frac{10}{11}B\right)\cdot (K_X+B)=0.
\end{equation}
Since $K_X+B$ is ample and $K_X+\frac{10}{11}B$ is pseudo-effective, we have $K_X+\frac{10}{11}B\equiv 0$. By the abundance for lc surface pairs, $K_X+\frac{10}{11}B\sim_{\mathbb Q}0$. 

(3) The dual graph of $\Supp f^{-1}(B)$ takes the form
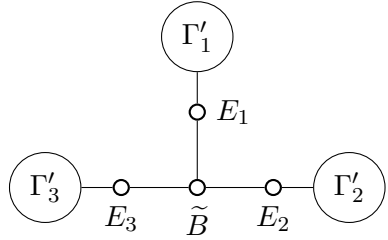
\[
\begin{tikzpicture}
    \node[wbullet, label=below: $\widetilde B$] (B) at (0,0){};
        \node[wbullet, label=right: $E_1$] (E1) at (0,1){};
    \node[draw, shape=ellipse] (Ga1) at (0,2){$\Gamma_1'$};
    \node[wbullet, label=below: $E_3$] (E3) at (-1, 0){};
    \node[draw, shape=ellipse](Ga3) at (-2, 0){$\Gamma_3'$};
    \node[wbullet, label=below: $E_2$] (E2) at (1, 0){};
        \node[draw, shape=ellipse] (Ga2) at (2, 0){$\Gamma_2'$};
    \draw (B)--(E1)--(Ga1);
    \draw (B)--(E2)--(Ga2);
    \draw (B)--(E3)--(Ga3);
\end{tikzpicture}
\]
where for each $1\leq i\leq 3$, $E_i$ is the component of $f^{-1}(x_i)$ intersecting $\widetilde B$, and $\Gamma_i$ is a (possibly empty) tree, so that $E_i\cup \Gamma_i$ is the dual graph of $f^{-1}(x_i)$. 

Let $e_i=-E_i^2$ and $q_i=\det(\Gamma_i)$. Then $e_i\geq 2$, and the quotient singularity $x_i\in X$ is of type $\frac{1}{n_i}(1, q_i)$. By \eqref{eq: 10/11B}, we have
\[
B^2=11\cdot\frac{1}{11}B\cdot(K_X+B) = 11(K_X+B)^2=\frac{11}{42}.
\]
and hence, as in \cite[Lemma~4.6]{LS23}, 
\[
\ZZ\owns\widetilde B^2 = B^2 - \sum_{1\leq i\leq 3}\frac{q_i}{n_i} = \frac{11}{42} - \frac{q_1}{2} - \frac{q_2}{3} - \frac{q_3}{7}.
\]
A simple enumeration gives us 
\[
(q_1, q_2, q_3) = (1, 1, 3),\, (e_1, e_2, e_3)=(2,3,3), \text{ and } \left(\widetilde B\right)^2=-1.
\]
Correspondingly, we have $\Gamma_1=\Gamma_2=\emptyset$,  $\Gamma_3$ is a chain of two $(-2)$-curves, and the dual graph of $f^{-1}(B)$ is as follows:
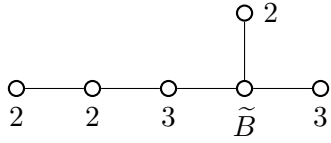
\[
\begin{tikzpicture}
    \node[wbullet, label=below: $\widetilde B$] (B) at (0,0){};
    \node[wbullet, label=below: $3$] (E2) at (1, 0){};
    \node[wbullet, label=right: $2$] (E1) at (0,1){};
    \node[wbullet, label=below: $3$] (E3) at (-1, 0){};
    \node[wbullet, label=below: $2$] (E31) at (-2, 0){};
    \node[wbullet, label=below: $2$] (E32) at (-3, 0){};
    \draw (B)--(E1);
    \draw (B)--(E2);
    \draw (B)--(E3)--(E31)--(E32);
\end{tikzpicture}
\]
\end{proof}

\begin{proof}[Proof of Theorem~\ref{thm: 462}]
We divide the proof into several steps.

\medskip

\noindent\textbf{Step 1}. In this step, we set up a ramified cover of $X$. 

Recall from Lemma \ref{lem: basic property log surface min vol} that $B$ is a smooth rational curve, $(X,B)$ is plt, $X$ is klt, and $K_X+\frac{10}{11}B\sim_{\mathbb Q}0$. Let $m$ be the minimal positive integer such that $$11m\left(K_X+\frac{10}{11}B\right)\sim 0.$$ This relation induces a ramified $(\ZZ/11m\ZZ)$-cover $\pi\colon \bar S\rightarrow X$ such that $\bar S$ has canonical singularities, $K_{\bar S}\sim 0$, and $B$ is the branch divisor of $\pi$.
By \cite[Corollary~2.51]{Kol13}, for any closed point $\bar s\in \bar S$, $$\pi\colon \left(\bar S\ni\bar s\right)\rightarrow \left(X\ni\pi(\bar s), \frac{10}{11}B\right)$$ is the index one cover.

Let $G$ be the Galois group $\Gal(\bar S/X)$, and $\sigma$ a generator of $G$. Then the stabilizer over $B$ is $\langle \sigma^m\rangle$, which has order 11. Let $G$ be the Galois group $\Gal(\bar S/X)$, and $\sigma$ a generator of $G$. Then the stabilizer over $B$ is $\langle \sigma^m\rangle$, which has order 11. Since $\bar S$ has canonical singularities and $K_{\bar S}\sim 0$, $\bar S$ is birationally a K3 surface or an abelian surface. But the latter case does not occur, since an abelian surface cannot have an automorphism of order $11$ fixing a point (\cite[Theorems 13.4.4 and 13.4.5]{BL04}). The surface $\bar S$ being birationally K3, we infer that $m\leq 6$ by \cite[Theorem 1.5]{OZ11}.

By the minimality of $m$, $G$ does not contain any non-trivial symplectic automorphism subgroup. In other words, the representation $$G\rightarrow \GL(H^0(\bar S, K_{\bar S}))\cong\mathbb{C}^*$$ is injective. 
 
Let $\pi''\colon \bar S\rightarrow\bar T=\bar S/\langle \sigma^{11}\rangle$ and $\pi'\colon \bar T\rightarrow X$ be the intermediate quotient maps, and $h\colon S\rightarrow\bar S$ the minimal resolution. Let $B_{\bar S}$ (resp.~$B_{\bar T}$, resp.~$B_S$) be the (reduced) preimages of $B$ in $\bar S$ (resp.~$\bar T$, resp.~$S$) respectively. We record the morphisms in the following commutative diagram:
\[
\begin{tikzcd}
S\arrow[r, "h"]& \bar S \arrow[rr, bend left = 40, "\pi"] \arrow[r, "\pi''"] &\bar T \arrow[r, "\pi'"] & X \\
B_S\arrow[u, symbol=\subset] \arrow[r] &B_{\bar S}\arrow[u, symbol=\subset]\arrow[r]& B_{\bar T}\arrow[u, symbol=\subset]\arrow[r] & B\arrow[u, symbol=\subset] \\
\end{tikzcd}
\] 
where
\begin{itemize}
\item $\pi$, $\pi'$ and $\pi''$ are finite morphisms of degrees $11m$, $11$ and  $m$ respectively,
\item  $\pi''$ is \'etale outside finitely many points of $\bar S$, and
\item  $\pi'$ is totally ramified over $B$.
\end{itemize}
We have the following easy observations:
\begin{itemize}
    \item The action of $G$ lifts to $S$.
    \item The restriction $h|_{B_S}\colon B_S\rightarrow B_{\bar S}$ of the resolution $h$ is birational.
    \item Since $\pi'$ is totally ramified over the smooth rational curve $B$, the restriction $\pi'|_{B_{\bar T}}\colon B_{\bar T}\rightarrow B$ is an isomorphism. 
\end{itemize}  
\medskip

\noindent\textbf{Step 2}.  In this step, we show that $m=1$ (hence $\bar S=\bar T$), and $S$ together with the action of $G$ is uniquely determined as in \cite[Theorem~7.3]{AST11}. 

By Lemma \ref{lem: basic property log surface min vol}, $B$ contains exactly three singular points $x_1,x_2,x_3$ of $X$, which are cyclic quotient singularities of order $n_1=2, n_2=3,n_3=7$ respectively. Since $\pi'^*(K_X+B) = K_{\bar T}+B_{\bar T}$ and $(X,B)$ is plt, $(\bar T,B_{\bar T})$ is plt. Let $\bar t_i\in \bar T$ be the inverse image of $x_i$ in $\bar T$ for $1\leq i\leq 3$. Then the Cartier index of $K_{\bar T}+B_{\bar T}$ near $\bar t_i$ (resp.~order of $\bar t_i\in \bar T$) is the same as that of $K_X+B$ near $x_i$ (resp.~order of $x_i\in X$), which is $n_i$ by \cite[Lemma 7.5]{LS23}. 

In particular, $\bar T$ is has an $A_1$-singularity at $\bar t_1$. It follows that $\pi''\colon \bar S\rightarrow\bar T$ is \'etale over $x_1$, and $\bar S$ has $A_1$-singularities over $x_1$. Note also that $K_{\bar T}$ is Cartier near $\bar t_3$, since otherwise the local degree of $\pi''$ over $\bar t_3$ would be $7$, exceeding the global degree $\deg\pi'' = m\leq 6$. Thus the induced morphism $B_{\bar S}\rightarrow B_{\bar T}$ is not ramified over $\bar t_3$. Therefore, $\pi$ is branched over at most two points on $B$, and the components of $B_{\bar S}$ and $B_S$ are still rational curves. In particular, the order $11$ non-symplectic automorphism $\sigma^m\in \Aut(S)$ fixes a rational curve. 

Then the surface $S$ together with the automorphism $\sigma^m$ are uniquely determined as in Lemma~\ref{lem: K3}, and $m\leq 2$ by \cite[Theorem 1.5]{OZ11}. Thus $K_{\bar T}$ is Cartier near $\bar t_2$, so $\bar S\rightarrow \bar T$ and hence $B_{\bar S}\rightarrow B_{\bar T}$ is not ramified over $\bar t_2$ for degree reasons again. Thus $\pi$ is branched over at most one point over $B$, and $B_{\bar S}$ and $B_S$ has $m$ irreducible components. By \cite[Theorem 7.3]{AST11}, $B_S$ has exactly one component. Therefore, $m=1$, and $\bar T\cong\bar S$.

\medskip

\noindent\textbf{Step 3}. In this step we show that the birational morphism $h\colon S\rightarrow \bar S$ is uniquely determined up to automorphisms of $S$ and $\bar S$.

By Step 2, the order of $\bar s_i\in \bar S$ is $n_i$ for $1\leq i\leq 3$. Since $\bar S$ has at most canonical singularities, $\bar s_i\in \bar S$ is of type $\frac{1}{n_i}(1, n_i-1)$. Therefore, $E_i:=h^{-1}(\bar s_i)$ is a chain of $n_i-1$ $(-2)$-curves for $1\leq i\leq 3$, which we view as a reduced Weil divisor, and the dual graph of $B_S+\sum_{1\leq i\leq 3} E_i$ is as follows:
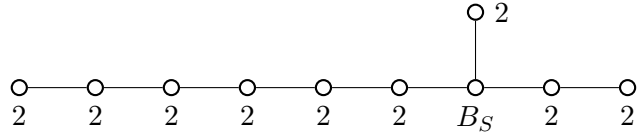
\[
\begin{tikzpicture}
    \node[wbullet, label=below: $B_S$] (B) at (0,0){};
    \node[wbullet, label=right: $2$] (E1) at (0,1){};
    \node[wbullet, label=below: $2$] (E21) at (1, 0){};
    \node[wbullet, label=below: $2$] (E22) at (2, 0){};
    \node[wbullet, label=below: $2$] (E31) at (-1, 0){};
    \node[wbullet, label=below: $2$] (E32) at (-2, 0){};
    \node[wbullet, label=below: $2$] (E33) at (-3, 0){};
    \node[wbullet, label=below: $2$] (E34) at (-4, 0){};
    \node[wbullet, label=below: $2$] (E35) at (-5, 0){};
    \node[wbullet, label=below: $2$] (E36) at (-6, 0){};
    \draw (B)--(E1);
    \draw (B)--(E21)--(E22);
    \draw (B)--(E31)--(E32)--(E33)--(E34)--(E35)--(E36);
\end{tikzpicture}
\]
where the smooth rational curve $B_S$ is also a $(-2)$-curve by the adjunction formula. By Lemma~\ref{lem: K3}, $(S, B_S+\sum_{1\leq i\leq 3} E_i)$ is uniquely determined up to  automorphism.

\medskip

\noindent\textbf{Step 4}. Note that $(K_{\bar S}+B_{\bar S}) = \pi^*(K_X+B)$ is ample, and $(\bar S, B_{\bar S})$ is the ample model of $(S, B_S+\sum_{1\leq i\leq 3} E_i)$. Hence the uniqueness of  $(\bar S, B_{\bar S})$ follows from that of $(S, B_S+\sum_{1\leq i\leq 3} E_i)$. Also the actions of the automorphism $\sigma$ on $S$ and $\bar S$ are uniquely determined. It follows that the quotient surface pair $(X, B)$ is uniquely determined, and hence must be the same as the surface pair $(U, B_U)$ in Example~\ref{ex: 462}.
\end{proof}

\begin{cor}\label{cor: normal}
Let $X$ be a stable surface that is not normal. Then $K_X^2>\frac{1}{462}$.
\end{cor}
\begin{proof}
Suppose on the contrary that $K_X^2\leq\frac{1}{462}$. Take the normalization $\nu\colon \bar X\rightarrow X$. Then the conductor divisor $\bar D\subset \bar X$ is nonempty, and $(\bar X, \bar D)=\sqcup_{1\leq i\leq m} (\bar X_i, \bar D_i)$ is a union of log canonical surfaces $(\bar X_i, \bar D_i)$ with $K_{\bar X_i}+\bar D_i$ ample and 
\[
\sum_{1\leq i\leq m}(K_{\bar X_i}+\bar D_i)^2 = K_X^2\leq \frac{1}{462}.
\]
By Theorems~\ref{thm: LS 462} and \ref{thm: 462}, we have $m=1$, $K_X^2=(K_{\bar X}+\bar D)^2= \frac{1}{462}$, and $\bar D$ is a smooth rational curve such that the different is $\Diff_{\bar D}(0)=\frac{1}{2}\bar x_1 + \frac{2}{3}\bar x_2 + \frac{6}{7} \bar x_3$, where the $\bar x_i$'s are the singular points of $\bar X$ lying on $\bar D$. The generically two-to-one morphism $\bar D\rightarrow \nu(\bar D)$ induces an involution $\tau$ on $\bar D$ preserving $\Diff_{\bar D}(0)$. But then $\tau$ fixes the three points $\bar x_i$, $1\leq i\leq 3$, and hence must be the identity, which is a contradiction.
\end{proof}

\section{The minimal volume of surfaces with \texorpdfstring{$K_X$}{} ample and \texorpdfstring{$\Nklt(X)\neq \emptyset$}{}}\label{sec: 825}
In this section, we find the minimal volume of surfaces $X$ with $K_X$ ample and $\Nklt(X)\neq \emptyset$. The proof consists of two major parts: The first part is the construction of the lc surface $X$ realizing the lower bound (Example~\ref{ex: 825}); the second more theoretic part  uses the technique of taking pseudo-effective thresholds and nef thresholds to prove the lower bound $\min\KK^2_\nklt\geq \frac{1}{825}$.  

\subsection{The surface \texorpdfstring{$V$}{} with \texorpdfstring{$K_V$}{} ample, \texorpdfstring{$\Nklt(V)\neq \emptyset$}{}, and \texorpdfstring{$K_V^2=\frac{1}{825}$}{}}

In this subsection, we construct the surface realizing the equality of Theorem~\ref{thm: 825}. The construction is parallel to Example~\ref{ex: 462}; in fact, the minimal resolution $\widetilde V$ of $V$ in Example~\ref{ex: 825} is obtained by further blowing up the surface $\widetilde U$ in Example~\ref{ex: 462}. For the convenience of the readers, we will spell out more details about this new example.
\begin{ex}[]\label{ex: 825}\label{ex: lc surface smallest volume}
 Let $L_0, L_1, L_2, L_3$ be four lines in general position on $\PP^2$. Let $g_V\colon \widetilde V\rightarrow \PP^2$ be the blow-up of some of the nodes of $\sum_{0\leq i\leq 3} L_i$, so that $g_V^{-1}(\sum_{0\leq i\leq 3} L_i)$ has the following dual graph
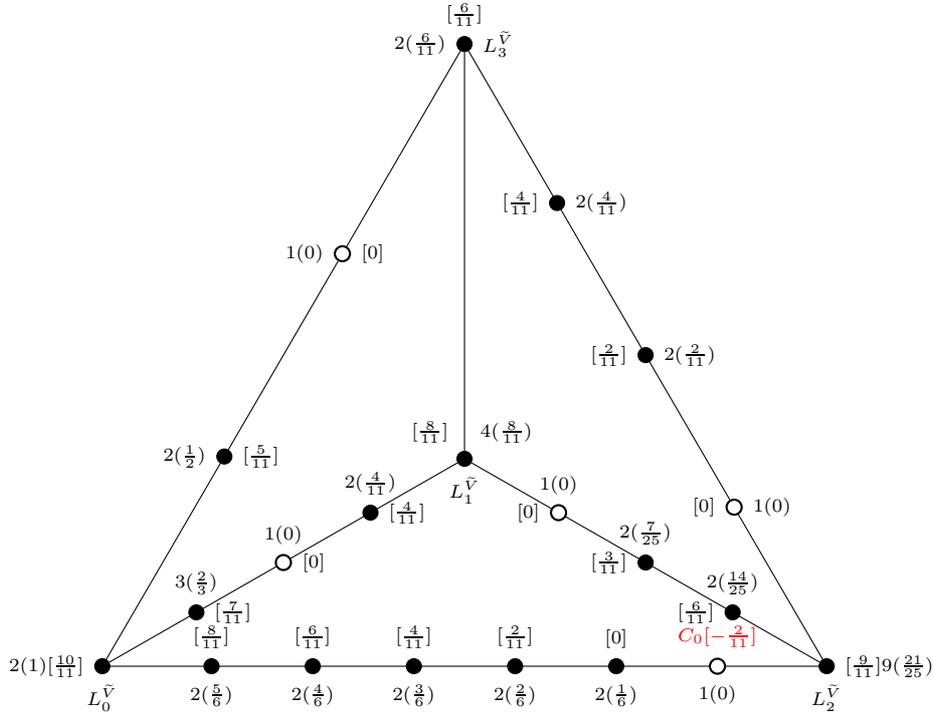
\begin{center}
\begin{tikzpicture}[font=\tiny, scale=5.5]
\node[bbullet, label=below:$L_0^{\widetilde V}$, label=left: $2(1){[\frac{10}{11}]}$] (L0) at (210:1) {};
\node[bbullet, label=below:$L_2^{\widetilde V}$,  label=right:${[\frac{9}{11}]}9(\frac{21}{25})$] (L2) at (330:1) {};
\node[bbullet, label=below:$L_1^{\widetilde V}$,  label=north east:$4(\frac{8}{11})$, label=north west:{$[\frac{8}{11}]$}](L1) at (0,0){};
\node[bbullet, label=right:$L_3^{\widetilde V}$, label=left:$2(\frac{6}{11})$, label=above: {$[\frac{6}{11}]$}] (L3) at (90:1) {};
\draw (L0) -- (L3)node[bbullet, pos=1/3, label=left:$2(\frac{1}{2})$, label=right:{$[\frac{5}{11}]$}](L03A){}node[wbullet, pos=2/3, label=left:$1(0)$, label=right:{$[0]$}](L03B){};
\draw (L0) -- (L1)
node[bbullet, pos=1/4, label=above:$3(\frac{2}{3})$, label=right:{$[\frac{7}{11}]$}](L02A){}
node[wbullet, pos=2/4, label=above: $1(0)$, label=right: {$[0]$}](L02B){}
node[bbullet, pos=3/4, label=above:$2(\frac{4}{11})$, label=right: {$[\frac{4}{11}]$}](L02C){};
\draw (L0) --(L2)
node[bbullet, pos=1/7, label=below:$2(\frac{5}{6})$, label=above:{$[\frac{8}{11}]$}](L01A){}
node[bbullet, pos=2/7, label=below:$2(\frac{4}{6})$, label=above:{$[\frac{6}{11}]$}](L01B){}
node[bbullet, pos=3/7, label=below:$2(\frac{3}{6})$, label=above:{$[\frac{4}{11}]$}](L01C){}
node[bbullet, pos=4/7, label=below:$2(\frac{2}{6})$, label=above:{$[\frac{2}{11}]$}](L01D){}
node[bbullet, pos=5/7, label=below:$2(\frac{1}{6})$, label=above:{$[0]$}](L01E){}
node[wbullet, pos=6/7, label=below:$1(0)$, label=above:{\color{red}$C_0[-\frac{2}{11}]$}](L01F){};
\draw (L2) -- (L1) node[bbullet, pos=1/4, label=above:$2(\frac{14}{25})$, label=left:{$[\frac{6}{11}]$}](L12A){}node[bbullet, pos=2/4, label=above:$2(\frac{7}{25})$, label=left:{$[\frac{3}{11}]$}](L12B){}node[wbullet, pos=3/4, label=above:$1(0)$, label=left:{$[0]$}](L12C){};
\draw (L1) -- (L3);
\draw (L2) -- (L3)node[wbullet, pos=1/4, label=right: $1(0)$, label=left:{$[0]$}](L23A){}node[bbullet, pos=2/4, label=right: $2(\frac{2}{11})$, label=left:{$[\frac{2}{11}]$}](L23B){}node[bbullet, pos=3/4, label=right:$2(\frac{4}{11})$, label=left:{[$\frac{4}{11}$]}](L23C){};
\end{tikzpicture}
\end{center}
where
\begin{itemize}
    \item $L_i^{\widetilde V}$ denotes the strict transforms of $L_i$ for $0\leq i\leq 3$;
    \item a white bullet denotes the exceptional $(-1)$-curve of the last blow-up over a node;
    \item a black bullet denotes either an $L_i^{\widetilde V}$ or a $(-e)$-curve with $e\geq 2$;
    \item the numbers outside the brackets denote the negative of the self-intersections of the corresponding curves; and
    \item the numbers in the square and round brackets denote the coefficients of the curve in the divisors $C_{\widetilde V}$ and $C_{\widetilde V}'$ (to be defined below) respectively.
\end{itemize} 
We will call the curves appearing on the above dual graph \emph{visible} curves on $\widetilde V$. Define the following $\QQ$-divisors on $\widetilde V$ supported on the visible curves:
\begin{itemize}
    \item $\widetilde C$ be the reduced divisor consisting of all the curves denoted by the black bullets.
    \item $C_{\widetilde V}$ is an effective $\QQ$-divisor with coefficients indicated in the round brackets.
    \item $C'_{\widetilde V}$ is a non-effective $\QQ$-divisor with coefficients indicated in the square brackets. There is exactly one curve, namely $C_0$, with negative coefficient in $C'_{\widetilde V}$. 
\end{itemize}
It is now straightforward to check the following properties:
\begin{enumerate}
    \item $C'_{\widetilde V} \leq C_{\widetilde V} \leq  \widetilde C$.
    \item $K_{\widetilde V} + C'_{\widetilde V} =g_V^*(K_{\PP^2}+ \frac{10}{11}L_0+\frac{8}{11}L_1 + \frac{9}{11}L_2 + \frac{6}{11}L_3)\sim_\QQ 0$. It follows that $K_{\widetilde V} + \widetilde C$ and $K_{\widetilde V} + C_{\widetilde V}$ are $\QQ$-linearly equivalent to the effective $\QQ$-divisors $\widetilde C-C'_{\widetilde V}$ and $C_{\widetilde V}-C'_{\widetilde V}$ respectively. 
    \item $K_{\widetilde V} + C_{\widetilde V}$ is nef since its intersection numbers with the visible curves are non-negative. In fact, for a visible curve $C$, one can check
    \[
    (K_{\widetilde V} + C_{\widetilde V})\cdot C 
    \begin{cases}
        =0 & \text{if $C$ is a black bullet}\\
        >0 & \text{if $C$ is not a black bullet}.
    \end{cases}
    \]
    \item $K_{\widetilde V} + C_{\widetilde V}$ is the positive part of $K_{\widetilde V} + \widetilde C$ in the Zariski decomposition.
    \item We can write $C_{\widetilde V}-C'_{\widetilde V} = \frac{2}{11} C_0 + D$, where $D$ is an effective $\QQ$-divisor supported on the visible curves corresponding to the black bullets. Therefore,
    \[
    (K_{\widetilde V} + C_{\widetilde V})^2 =(K_{\widetilde V} + C_{\widetilde V})\cdot\left(\frac{2}{11}C_0+D\right) = (K_{\widetilde V} + C_{\widetilde V})\cdot \left(\frac{2}{11}C_0\right) = \frac{2}{11}\left(\frac{1}{6}+\frac{21}{25}-1\right) =\frac{1}{825}.
    \]
    \item Since the visible curves form a simple normal crossing divisor on $\widetilde V$, the pairs $(\widetilde V, \widetilde C)$ and $(\widetilde V, B_{\widetilde V})$ are both lc. 
\end{enumerate}
Now take the ample model $\pi\colon \widetilde V\rightarrow V$ of $K_{\widetilde V}+ C_{\widetilde V}$. Since $K_{\widetilde V}+ C_{\widetilde V}$ is big and nef, and $(K_{\widetilde V}+ C_{\widetilde V})\cdot C_{\widetilde V}=0$, the boundary divisor $C_{\widetilde V}$ is contracted by $\pi$, and $K_V = \pi_*(K_{\widetilde V}+ C_{\widetilde V})$ is ample. We have $\pi^* K_V = K_{\widetilde V}+ C_{\widetilde V}$ and by (5), we have
\[
K_V^2 = (K_{\widetilde V}+ C_{\widetilde V})^2 = \frac{1}{825}
\]
The whole construction can be illustrated by the following diagram:
\[
\begin{tikzcd}
    (\widetilde V, \widetilde C) \arrow[r, symbol = \supset]\arrow[d, "g_V"']& (\widetilde V, C_{\widetilde V})  \arrow[d, "\pi"] \\
    (\PP^2, \sum_{i=0}^{3} L_i) & V
\end{tikzcd}
\]
Since exactly one component of $C_{\widetilde V}$, namely $L_0^{\widetilde V}$, has coefficient $1$, we infer that $X$ has a unique lc center $P_0=\pi(L_0^{\widetilde V})$ and two other non-canonical singularities $P_1=\pi_*(L_1^{\widetilde V})$ and $P_2=\pi_*(L_2^{\widetilde V})$, where $P_1$ and $P_2$ are quotient singularities of types $\frac{1}{22}(1,13)$ and $\frac{1}{25}(1,3)$ respectively. One can check as in \cite[Theorem 5.2]{AL19a} that $\pi$  contracts exactly visible curves corresponding to the black bullets of the dual graph. It follows that the Picard number $\rho(V)=2$, and $P_0, P_1, P_2$ are the only singularities of $V$.
\end{ex}

Recall that the canonical ring of a stable surface $X$ is given by $R(X, K_X):=\bigoplus_{n\geq 0} H^0(X, nK_X)$, which is finitely generated. It provides an embedding of $X$ into the weighted projective space $\Proj R(X,K_X)$, called the \emph{canonical embedding} in this paper. We utilized the Riemann--Roch theorem for singular surfaces, as presented by Blache in \cite{Bla95}, to compute the plurigenera of $V$. From our calculations, we figure out that $V$ is possibly a hypersurface of degree $86$ in the weighted projective space $\PP:=\PP(6,11,25,43)$. The following theorem confirms this hypothesis by examining the linear system $|\sO_{\PP}(86)|$, and it also sheds light on the moduli space $M_{\frac{1}{825}}$ of stable surfaces $X$ with $K_X^2=\frac{1}{825}$. 
\begin{thm}\label{thm: WPS86}
Let $x_0, x_1, x_2, x_3$ be the homogeneous coordinates of $\PP$ of degrees $6,11,25,43$ respectively. Then the following holds.
\begin{enumerate}
    \item Any hypersurface of degree $86$ in $\PP$ is projectively equivalent to some $V_{\vec{\epsilon},(s,t)}$ defined by
\[
H_{\vec{\epsilon}, (s,t)}:=\epsilon_1 x_3^2+ \epsilon_2 x_3 x_2 x_0^3+\epsilon_3 x_2^3x_1+ sx_2^2 x_0^6 +\epsilon_4 x_2x_1^5 x_0 +tx_1^4 x_0^7 =0
\]
with $\epsilon_i\in\{0,1\}$ for $1\leq i\leq 4$ such that $\epsilon_1\epsilon_2=0$ and $s, t\in \CC$.
\item $V_{\vec{\epsilon},(s,t)}$ is lc if and only if $\epsilon_1=\epsilon_3=\epsilon_4=1$,  $\epsilon_2=0$, and $s,t$ are not both zero. More precisely, setting $V_{(s,t)}:=V_{(1,0,1,1), (s,t)}$, the following holds.
\begin{enumerate}
    \item The canonical indices of $P_0, P_1, P_2$ in $V_{(s,t)}$ are $6,11,25$ respectively, and $V_{(s,t)}$ has at most $A_1$-singularities elsewhere;
    \item $V_{(s,t)}$ is klt if and only if $s\neq 0$.
\end{enumerate} 
\item For $(s,t)\in \CC\setminus\{(0,0)\}$, the two hypersurfaces $V_{(s,t)}$ and $V_{(s',t')}$ are projetively equivalent if and only if $(s:t)=(s':t')$ as points in $\PP^1$. 
\item For $(s,t)\in \CC\setminus\{(0,0)\}$, there is an isomorphism of graded $\CC$-algebras:
\[
R(V_{(s,t)}, K_{V_{(s,t)}}) \cong \CC[x_0, x_1, x_2, x_3]/(H_{(s,t)}).
\]
where $H_{(s,t)} = H_{(1,0,1,1), (s,t)}$ is the defining polynomial of $V_{(s,t)}$.
\item The induced morphism to the moduli space $\PP^1_{(s:t)}\rightarrow M_{\frac{1}{825}}$ is a homeomorphism onto the image.
\end{enumerate}
\end{thm}
\begin{proof}
(1) A weighted homogeneous polynomial of degree $86$ in $\CC[x_0, x_1, x_2, x_3]$ can be written as 
\[
H = a_1 x_3^2+ a_2  x_3 x_2 x_0^3 + a_3 x_2^3x_1+a_4 x_2^2 x_0^6 +a_5 x_2x_1^5 x_0 +a_6 x_1^4 x_0^7
\]
with $a_i\in \CC$ are not all zero. An automorphism $\sigma\in \Aut(\PP)$ can be written as
\[
\sigma(x_0, x_1, x_2, x_3) = (c_0 x_0, c_1 x_1, c_2 x_2, c_3 x_3  + d x_2 x_0^3)
\]
with $c_i\in \CC\setminus\{0\}$ for $0\leq i\leq 3$ and $d\in \CC$, and we also use $\sigma_{(\Vec{c},d)}$ to denote this $\sigma$. Two hypersurfaces $\{H=0\}$ and $\{H'=0\}$ are projectively equivalent if for some $\sigma\in\Aut(\PP)$, 
\[
H'(x_0, x_1, x_2, x_3)=(\sigma^*H)(x_0, x_1, x_2, x_3):=H\circ\sigma(x_0, x_1, x_2, x_3).
\]  
If $a_1\neq 0$, then one can use the automorphism $(x_0, x_1, x_2, x_3)\mapsto (x_0, x_1, x_2, x_3-\frac{a_2}{2a_1}x_2x_0^3)$ to eliminate the term $a_2x_3x_2x_0^3$ in $H$. Then it is straightforward to check that there is $(c_0, c_1, c_2, c_3)\in (\CC\setminus\{0\})^4$ such that 
\[
\sigma_{(\vec{c},0)}^*H = \lambda(\epsilon_{1} x_3^2+\epsilon_2x_3x_2x_0^3 + \epsilon_{3}   x_2^3x_1+ sx_2^2 x_0^6 +\epsilon_{4}  x_2x_1^5 x_0 +t x_1^4 x_0^7)
\]
with $\epsilon_i\in \{0,1\}$ for $1\leq i\leq 4$ satisfying $\epsilon_1\epsilon_2=0$, $t\in \CC$ and $\lambda\in \CC\setminus\{0\}$. 

\medskip

(2) Since the weights $n_1=6, n_2=11, n_3=25, n_4=43$ are prime to each other, $\PP$ is well-formed, with exactly the four coodinate points $P_i=\{x_j=x_k=x_l=0\}$ as singularities, where $\{i,j,k,l\}=\{0,1,2,3\}$ (\cite[5.15]{IF00}). Note that $P_i\in V_{\vec{\epsilon},(s,t)}$ for $0\leq i\leq 2$, and $P_3 \in V_{\vec{\epsilon},(s,t)}$ if and only if $\epsilon_1=0$. 

For $0\leq i\leq 3$, let $H_{\vec{\epsilon},(s,t)}^{i}$ be the polynomial in $x_j, x_k, x_l$ obtained by setting $x_i=1$ in $H_{\vec{\epsilon},(s,t)}$, where $\{i,j,k,l\}=\{0, 1,2,3\}$. Then
\[
S^i_{\vec{\epsilon},(s,t)} :=\{H_{\vec{\epsilon},(s,t)}^i=0\} \subset \CC^3
\]
is a hypersurface with affine coordinates $x_j, x_k, x_l$, and $V_{\vec{\epsilon},(s,t)}\cap\{x_i\neq 0\}$ is obtained as the quotient of $(S^i_{\vec{\epsilon},(s,t)}\ni 0)$ by a $(\ZZ/n_i\ZZ)$-action that is free away from the origin. Moreover, $0\in S^i_{\vec{\epsilon},(s,t)}$ if and only if $P_i\in V_{\vec{\epsilon},(s,t)}$. By \cite[Proposition 5.20]{KM98}, $(V_{\vec{\epsilon},(s,t)}\ni P_i)$ is lc (resp.~klt) if and only if $(S^i_{\vec{\epsilon},(s,t)}\ni 0)$ is so. 

If $\epsilon_1 =0$, then $(V_{\vec{\epsilon},(s,t)}\ni P_3)$ is the quotient of the hypersurface singularity $(S^3_{\vec{\epsilon},(s,t)}\ni 0)$, where $S^3_{\vec{\epsilon},(s,t)}=\{\epsilon_2 x_3 x_2 x_0^3+\epsilon_3 x_2^3x_1+ sx_2^2 x_0^6 +\epsilon_4 x_2x_1^5 x_0 +t x_1^4 x_0^7 = 0\}\subset\CC^3$. Since the multiplicity of  the hypersurface $S^3_{\vec{\epsilon},(s,t)}$ at $0$ is at least $4$, it cannot be lc. It follows that $(V_{\vec{\epsilon},(s,t)}\ni P_3)$ is not lc. Similarly, if $\epsilon_3 = 0$ (resp.~$\epsilon_4=0$, resp.~$s=t=0$), then $(V_{\vec{\epsilon},(s,t)}\ni P_i)$ is not lc for $i=2$ (resp.~$i=1$, resp.~$i=0$).

Now suppose that $\epsilon_1=\epsilon_3=\epsilon_4=1$,  $\epsilon_2=0$, and $s,t$ are not both zero. In this case, for simplicity of notation, we denote 
\[
V_{(s,t)}= V_{\vec{\epsilon},(s,t)},\,\, 
H_{(s,t)} =H_{\vec{\epsilon},(s,t)},\,\,
S^i_{(s,t)} = S^i_{\vec{\epsilon},(s,t)},\,\, H^i_{(s,t)} =H^i_{\vec{\epsilon},(s,t)}. 
\]
where $0\leq i\leq 3$. We need to show that $V_{(s,t)}$ is lc. Recall that $V_{(s,t)}$ is now defined by
\[
x_3^2+ x_2^3x_1+ sx_2^2 x_0^6 + x_2x_1^5 x_0 +tx_1^4 x_0^7 = 0,\,\,(s,t)\in \CC^2\setminus\{(0,0)\}
\]
and thus $P_3\notin V_{(s,t)}$. 

If $s=0$, then we may assume that $t=1$ by applying an appropriate diagonal projective transformation. It can be checked as above that $(S^0_{(0,1)}\ni 0)$ is a hypersurface elliptic singularity of type $\widetilde E_7$ (see \cite[Theorem~7.6.4]{Ish18}), $(S^1_{(0,1)}\ni 0)$ an $A_1$-singularity,  $(S^2_{(0,1)}\ni 0)$ smooth. It follows that $V_{(0,1)}$ is lc at $P_i$ for $0\leq i\leq 2$, and non-klt at $P_0$. One can use the Jacobian criterion, by hand or with the help of computer algebra such as SAGE, to see that $V_{(0,1)}$ does not have any other singularities.

If $s\neq 0$, then one can check as before that $(S^i_{(s,t)}\ni 0)$ is canonical for $0\leq i\leq 2$, and it follows that $(V_{(s,t)}\ni P_i)$ is klt. It is more involved to check that $V_{(s,t)}\setminus\{P_0, P_1, P_2\}$ is klt, but the idea is similar. In fact, we can prove a stronger result:
\begin{claim}
If $s\neq 0$, then $V_{(s,t)}$ has at most $A_1$-singularities away from the set $\{P_0,P_1,P_2\}$. 
\end{claim}
\begin{proof}[Proof of the claim.]
For $0\leq i\leq 2$, the morphism $S^i_{(s,t)}\rightarrow V_{(s,t)}\cap \{x_i\neq 0\}$
is \'etale away from $P_i$. Thus it suffices to show that $S^i_{(s,t)}\setminus\{0\}$ has at most $A_1$-singularities. For this, consider the double cover 
\[
\pi_i\colon S^i_{(s,t)}\rightarrow \CC^2, \, (x_j, x_k, x_3)\mapsto (x_j, x_k)
\]
where $\{i,j,k\}=\{0,1,2\}$. The branch curve 
of $\pi_i$ is given by $B_i=\{H^i_{(s,t)}(x_j, x_k, 0)=0\}$. With the help of SAGE or by hand, one can check that the zeros of the following ideal
\[
J_i:=\left(H^i_{(s,t)}(x_j, x_k, 0), \frac{\partial H^i_{(s,t)}}{\partial x_j}(x_j, x_k, 0), \frac{\partial H^i_{(s,t)}}{\partial x_k}(x_j, x_k, 0), \Hess\left(H^i_{(s,t)}(x_j, x_k, 0)\right)\right)
\]
are contained in $\{(0,0)\}$, where the last generator of $J_i$ is the Hessian of $H^i_{(s,t)}(x_j, x_k, 0)$ with respect to $x_j$ and $x_k$. This means that $B_i$ has at most nodes as singularities away from $(0,0)$, and hence the double cover $S^i_{(s,t)}$ has at most $A_1$-singularities away from the origin. 
\end{proof}

Note that for any $(s,t)\in \CC^2\setminus\{(0,0)\}$ and $0\leq i\leq 2$, $(S^i_{(s,t)}\ni 0)$ is the canonical cover of $(V_{(s,t)}\ni P_i)$, and thus the canonical index of $(V_{(s,t)}\ni P_i)$ is $n_i$. 
This finishes the proof of the assertions in (2) of the theorem.

\medskip

(3) Projective transformations sending $V_{(s',t')}$  to $V_{(s, t)}$ must be of the diagonal form, that is, of the form $\sigma(x_0, x_1, x_2, x_3)=(c_0 x_0, c_1 x_1, c_2 x_2, c_3 x_3)$ for some $c_0, c_1, c_2, c_3\in \CC\setminus\{0\}$. Then $
\sigma(V_{(s', t')}) = V_{(s,t)}$ if and only if $\sigma^*H_{(s,t)}=\lambda H_{(s',t')}$ for some $\lambda\in \CC\setminus\{0\}$, that is, comparing the coefficients of each monomial on both sides, we have the following equations:
\[
c_3^2 = c_2^3c_1 = c_2 c_1^5 =\lambda,\,\, c_2^2c_0^6s=\lambda s', \,\, c_1^4c_0^7 t= \lambda t'
\]
which gives $s=\frac{\lambda^{25}}{c_2^{86}}s'$ and $t=\frac{\lambda^{25}}{c_2^{86}}t'$ with $\lambda$ and $c_2$ freely chosen. The assertion (3) follows.

\medskip

(4) The weighted projective space $\PP$ and the hypersurface $V_{(s,t)}$ are both well-formed (\cite[6.10]{IF00}). Thus we may apply the adjunction formula to get 
\[
\sO_{V_{(s,t)}}(K_{V_{(s,t)}}) =\sO_{V_{(s,t)}}(86-6-11-25-43) = \sO_{V_{(s,t)}}(1).
\]
By the short exact sequence
\[
0\rightarrow \sO_{\PP}(n-86) \xrightarrow{H_{(s,t)}} \sO_{\PP}(n)\rightarrow  \sO_{V_{(s,t)}}(n)\rightarrow 0
\]
and the vanishing $H^1(\PP, \sO_{\PP}(n-86)) =0$ for each $n\in \ZZ$ (\cite{Dol81}), we obtain the required isomorphism of graded $\CC$-algebras
\[
R\left(V_{(s,t)}, K_{V_{(s,t)}}\right)\cong \bigoplus_{n\geq 0} H^0\left(V_{(s,t)}, \sO_{V_{(s,t)}}(n)\right)\cong\CC[x_0, x_1, x_2, x_3]/\left(H_{(s,t)}\right)
\]

\medskip

(5) By (2), each $V_{(s,t)}$ is lc for $(s,t)\in \CC^2\setminus\{(0,0)\}$ . Since $\sO_{
V(s,t)}(K_{V_{(s,t)}}) = \sO_{V_{(s,t)}}(1)$ is ample, $V_{(s,t)}$ is a stable surface with volume 
\[
K_{V_{(s,t)}}^2 = \frac{86}{6\cdot 11\cdot 25\cdot 43} =\frac{1}{825}.
\]
Thus we have a morphism $\Phi\colon\CC^2\setminus\{(0,0)\}\rightarrow M_{\frac{1}{825}}$, $(s,t)\mapsto V_{(s,t)}$ into the moduli space. By (3), $\Phi$ descends to a one-to-one morphism $\varphi\colon \PP^1\rightarrow M_{\frac{1}{825}}$. Since bijective morphisms between curves are homeomorphisms, the assertion (5) follows.
\end{proof}

\begin{cor}
    The surface $V$ constructed in Example~\ref{ex: 825} is isomorphic to the surface $V_{(0,1)}$ of Theorem~\ref{thm: WPS86}. As a consequence, $V$ can be embbeded into $\PP(6,11,25,43)$ as a hypersurface of degree $86$, with defining equation $x_3^2+ x_2^3x_1 + x_2x_1^5 x_0 +x_1^4 x_0^7 = 0$.
\end{cor}
\begin{proof}
The surfaces $V$ and $V_{(0,1)}$ are both stable surfaces with non-klt locus and volume $\frac{1}{825}$. By the uniqueness part of Theorem~\ref{thm: 825}, we have $V\cong V_{(0,1)}$.
\end{proof}

\subsection{Proof of the main theorem}\label{sec: pf main thm}

Now we turn to the proof of Theorem~\ref{thm: 825}, which is  divided into several steps.

\medskip

\noindent\textbf{Step 1}. In this step, we introduce two birational models $Y$ and $Z$ of $X$, on which we can decrease the number of lc places related to $K_X$ in a managable way. They sit in the following diagram:
\[
\begin{tikzcd}
    &Y\arrow[ld, "f"']\arrow[rd, "g"]& \\
   X&& Z
\end{tikzcd}
\]
Let $\{x_1,\dots,x_n\}$ be the set of lc centers of $X$. By Lemma \ref{lem: e^2 and -1/3}, we can take a projective birational morphism $f: Y\rightarrow X$ extracting only divisors with log discrepancy 0. Morever, setting $E_i=f^{-1}(x_i)$ for $1\leq i\leq n$ and $E=\sum E_i$, we may assume that the following holds.
\begin{enumerate}
    \item $Y$ is klt.
    \item $-E$ is ample$/X$.
    \item $E_i^2\leq -\frac{1}{3}$ for each $1\leq i\leq n$. Moreover, $E_i^2=-\frac{1}{3}$ if and only if the dual graph $\mathcal{D}(X\ni x_i)$ is as in Table \ref{table: two special singularity} of Lemma~\ref{lem: e^2 and -1/3}, and $E_i$ is the image of the unique fork of $\mathcal{D}(X\ni x_i)$.
\end{enumerate}
It follows that $f^*K_X= K_Y+E$, and
\[
E^2=\sum_{i=1}^nE_i^2\leq -\frac{1}{3},
\]
and equality holds if and only if $n=1$, $E=E_1$, $\mathcal{D}(X\ni x_1)$ is as in Table \ref{table: two special singularity} of Lemma~\ref{lem: e^2 and -1/3}, and $E_1$ is the image of the unique fork of $\mathcal{D}(X\ni x_1)$.

Now that $K_Y+E$ is big and nef, we may define 
$$t:=\inf\{s\geq 0\mid K_Y+sE\text{ is nef}\}.$$
Since $K_X$ is ample and $-E$ is ample$/X$, we have $t<1$. Since $(K_Y+E)\cdot E=0$, we have
\begin{equation}\label{eq: bound vol 1}
    K_X^2=(K_Y+E)^2=(K_Y+tE)^2+(1-t)^2(-E^2)\geq (1-t)^2(-E^2)\geq\frac{1}{3}(1-t)^2.
\end{equation}
We may assume that 
\[
1-t\leq\sqrt{\frac{1}{275}}<\frac{1}{16},
\]
since otherwise we obtain $K_X^2> \frac{1}{825}$ from \eqref{eq: bound vol 1}.  In particular, $t>0$. Since $(Y,tE)$ is lc and $K_Y+tE$ is nef, $K_Y+tE$ is semi-ample and hence defines a contraction $g: Y\rightarrow Z$ such that  $K_Y+tE\equiv_Z0$. Since $(K_Y+tE)\cdot E>0$,  $E$ is horizontal$/Z$. Since $t>\frac{15}{16}$, $g$ is birational by \cite[Corollary 7.1]{LS23}. Set $E_Z:=g_*E$. Then $K_Z+tE_Z$ is ample.

\medskip

\noindent\textbf{Step 2}. In this step we study the structure of $(Z,E_Z)$ by proving the following two claims.

\begin{claim}\label{claim: K+E positive}
The exceptional loci $\Exc(f)$ and $\Exc(g)$ do not have any common component.    For each prime divisor $C\subset \Exc(g)$, both of $(K_Y+E)\cdot C$ and $E\cdot C$ are positive.
\end{claim}
\begin{proof}[Proof of the claim.]
    For each component $C\subset \Exc(f)$, we have 
    \[
    (K_Y+tE)\cdot C = (K_Y+E)\cdot C + (t-1)E\cdot C =(t-1)E\cdot C >0
    \]
    where the last inequality is because of the relative ampleness of $-E$ over $X$. On the other hand, the components $C'$ of $\Exc(g)$ is characterized by the property $(K_Y+tE)\cdot C'=0$. Therefore, $C\neq C'$, and it follows that $\Exc(f)$ and $\Exc(g)$ do not have any common component. 
    
    For a component $C'\subset \Exc(g)$, we have
    \[
    (K_Y+E)\cdot C' = K_X\cdot f_*C'>0
    \]
   Combined with the fact that $(K_Y+tE)\cdot C'=0$, we obtain $E\cdot C'>0$. 
\end{proof}

\begin{claim}\label{claim: basic property z}
    We have the following:
    \begin{enumerate}
        \item $Z$ is klt.
        \item $(Z,E_Z)$ is lc.
        \item $K_Z+E_Z$ is big and nef.
        \item $(K_Z+E_Z)\cdot E_Z>0$.
        \item $(K_Z+E_Z)\cdot D>0$ for any curve $D\not\subset E_Z$.
    \end{enumerate}
\end{claim}
\begin{proof}[Proof of the claim.]
    (1) Since $(Y,tE)$ is klt and $K_Y+tE=g^*(K_Z+tE_Z)$, we infer that $(Z,tE_Z)$ and hence $Z$ is klt.

(2) Since $t>\frac{15}{16}$ and $(Z,tE_Z)$ is lc, by \cite{Kuw99}, \cite[Corollary 3.3]{Pro02}, $(Z,E_Z)$ is also lc.

(3) Since $K_Y+E$ is big and nef, $K_Z+E_Z=g_*(K_Y+E)$ is big and nef. 

(4) By Claim~\ref{claim: K+E positive} and the negativity lemma, we have $g^*(K_Z+E_Z) = K_Y+E+ G$, where $\Supp(G)=\Exc(g)$. Therefore,
\[
(K_Z+E_Z)\cdot E_Z=g^*(K_Z+E_Z) \cdot E= (K_Y+E+G)\cdot E =G\cdot E.
\]
By Claim \ref{claim: K+E positive}, $G\cdot E_Z>0$. Thus $(K_Z+E_Z)\cdot E_Z>0$.

(5) For any curve $D\not\subset E_Z$, since $K_Z+tE_Z$ is ample, we have
$$(K_Z+E_Z)\cdot D\geq (K_Z+tE_Z)\cdot D>0.$$
\end{proof}

\noindent\textbf{Step 3}. In this step, we prove the lower bound $K_X^2\geq \frac{1}{825}$, and provide necessary conditions for the equality case.

By Claim \ref{claim: basic property z}(1-3), $c:=\pet(Z,0;E_Z)$ is well-defined and satisfies $c<1$. By Theorem~\ref{thm: gap pet}, $c\leq\frac{10}{11}$ or $c=\frac{12}{13}$. We have
\begin{align*}
   (K_Y+E)^2&=(K_Y+tE)^2+(1-t)^2(-E^2)\\
   &=(K_Z+tE_Z)^2+(1-t)^2(-E^2)\\
   &=\left(\frac{t-c}{1-c}(K_Z+E_Z)+\frac{1-t}{1-c}(K_Z+cE_Z)\right)^2+(1-t)^2(-E^2)\\
   &\geq\left(\frac{t-c}{1-c}\right)^2(K_Z+E_Z)^2+(1-t)^2(-E^2).
\end{align*}
By Lemma \ref{lem: e^2 and -1/3},
$$(K_Y+E)^2\geq\left(\frac{t-c}{1-c}\right)^2(K_Z+E_Z)^2+\frac{1}{3}(1-t)^2,$$
and the inequality holds only if $n=1$, $E=E_1$, $\mathcal{D}(X\ni x_1)$ is as in Table \ref{table: two special singularity}, and $E_1$ is the image of the unique fork of $\mathcal{D}(X\ni x_1)$. There are two cases:

\medskip

\noindent\textbf{Case 1}. $c=\frac{12}{13}$. In this case, by Proposition \ref{prop: qfact 12/13 smallest volume}, $(K_Z+E_Z)^2\geq\frac{1}{260}$. Therefore,
$$(K_Y+E)^2\geq \frac{1}{260}\left(\frac{t-c}{1-c}\right)^2+\frac{1}{3}(1-t)^2=\frac{1}{260}(13t-12)^2+\frac{1}{3}(1-t)^2.$$
The quadratic function $$\frac{1}{260}(13t-12)^2+\frac{1}{3}(1-t)^2$$ has minimum $\frac{1}{767}$ (achieved at $t=\frac{56}{59}$). Therefore, $$(K_Y+E)^2\geq\frac{1}{767}>\frac{1}{825}$$ and we are done.

\medskip

\noindent\textbf{Case 2}. $c\leq\frac{10}{11}$. In this case, we let $\phi: Z\rightarrow T$ be the ample model of $K_Z+E_Z$ and let $E_T:=\phi_*E_Z$. Then $K_T+E_T$ is ample and $\phi^*(K_T+E_T)=K_Z+E_Z$. In particular, $(T,E_T)$ is lc. By Claim \ref{claim: basic property z}(4), $(K_Z+E_Z)\cdot E_Z>0$, so $E_Z$ is not contracted by $\phi$. Thus $E_T\not=0$.  By \cite [Theorem 1.4]{LS23}, $$(K_Z+E_Z)^2=(K_T+E_T)^2\geq\frac{1}{462}.$$ 
Therefore,
$$K_X^2=(K_Y+E)^2\geq \frac{1}{462}\left(\frac{t-c}{1-c}\right)^2+\frac{1}{3}(1-t)^2\geq\frac{1}{462}(11t-10)^2+\frac{1}{3}(1-t)^2,$$
and the equality holds only if $n=1$, $E=E_1$, $\mathcal{D}(X\ni x_1)$ is as in Table \ref{table: two special singularity} of Lemma~\ref{lem: e^2 and -1/3}, $E_1$ is the image of the unique fork of $\mathcal{D}(X\ni x_1)$, and $(K_Z+E_Z)^2=\frac{1}{462}$. By Theorem \ref{thm: 462}, the equality $(K_T+E_T)^2=\frac{1}{462}$ implies that $(T,E_T)$ is isomorphic to the surface pair $(U, B_U)$ presented in Example~\ref{ex: 462}. In particular, $(T,E_T)$ is plt.  By Claim \ref{claim: basic property z}(5), $(K_Z+E_Z)\cdot D>0$ for any prime divisor $D\not\subset\Supp E_Z$. Thus $\phi$ only contracts some irreducible components of $E_Z$. In particular, for any prime divisor $D$ contracted by $\phi$, we have
$$a(D,T,E_T)=a(D,Z,E_Z)=0.$$
Since $(T,E_T)$ is plt, $\phi$ is an isomorphism, and $(Z,E_Z)=(T,E_T)$.

The quadratic function $$\frac{1}{462}(11t-10)^2+\frac{1}{3}(1-t)^2$$ has minimum $\frac{1}{825}$ (achieved at $t=\frac{24}{25}$). Therefore, $K_X^2\geq\frac{1}{825}$, and equality holds only if
\begin{itemize}
    \item there is exactly one lc center $x_1\in X$,
    \item $\mathcal{D}(X\ni x_1)$ is as in Table \ref{table: two special singularity} of Lemma~\ref{lem: e^2 and -1/3},
    \item $E_1$ is the unique fork of $\mathcal{D}(X\ni x_1)$, 
    \item $(Z,E_Z)$ is isomorphic to the surface pair $(U, B_U)$ in Example~\ref{ex: 462}, and
    \item $t=\frac{24}{25}$.
\end{itemize}

\noindent\textbf{Step 4}. In this step, we show the uniqueness of $X$ with $K_X^2=\frac{1}{825}$.

We let $y_1,y_2,y_3$ be the three singular points of $Y$ on $E$, such that the order of $y_1$ (resp.~$y_2,y_3$) is $2$ (resp.~$3,\,6$), and let $z_i=g(y_i)$ be the image on $Z$. Since $g$ is $(K_Y+E)$-positive by Claim~\ref{claim: K+E positive}, we have the comparison $\ord(Z\ni z_i)\geq \ord(Y\ni y_i)$ for each $1\leq i\leq 3$ by Lemma~\ref{lem: index of positive contraction}. By Lemma \ref{lem: basic property log surface min vol}, $E_Z$ is a smooth rational curve, and the orders of $z_1,z_2,z_3$ are equal to $2,3,7$ respectively. By Lemma \ref{lem: index of positive contraction}, $\mathcal{D}(X\ni x_1)$ is of type 2 as in Table \ref{table: two special singularity}. Since $\ord(Z\ni z_i) = \ord(Y\ni y_i)$ for $i\in\{1,2\}$ and $\ord(Z\ni z_3)=\ord(Y\ni y_3)+1$, we know that $g$ is an isomorphism near $y_1,y_2$, and $g$ contracts exactly one rational curve $B$ through $y_3$ by Lemma~\ref{lem: index of positive contraction}. Since $g$ is $E$-positive, $g$ is an isomorphism outside the inverse image of $E_Z$.

Since $(Z,E_Z)$ is uniquely determined as in Example~\ref{ex: 462} and $a(B,Z,\frac{24}{25}E_Z)=1$, there are only finitely many possibilities of $B$. By an easy enumeration, we see that there is only one possible choice of $B$. Therefore, as the extraction of $B$ from $(Z, E_Z)$, the lc surface $(Y,E)$ is unique. It follows that $X$, the canonical model of $(Y, E)$, is unique, and hence  must be the surface $V$ constructed in Example \ref{ex: lc surface smallest volume}.
\qed

\subsection{Some variants and applications of the main results}
We provide several variants and applications of the main results, which might be useful in practice.

\begin{thm}\label{thm: general version lc pair vol min}
    Let $(X,B)$ be a projective lc surface pair such that $B$ is  reduced, $K_X+B$ is big and nef, and $(X,B)$ is not klt. Then $\vol(K_X+B)\geq\frac{1}{825}$, and equality holds if and only if its ample model is isomorphic to the surface $V$ constructed in Example \ref{ex: lc surface smallest volume}.
\end{thm}
\begin{proof}
Since $(X,B)$ is lc, $K_X+B$ is semi-ample. Let $f: X\rightarrow X'$ be the ample model of $K_X+B$ and let $B':=f_*B$. Then $\vol(K_X+B)=\vol(K_{X'}+B')$, $B'$ is an effective Weil divisor, $K_{X'}+B'$ is ample, and $(X',B')$ is not klt. If $B'\not=0$, then by \cite[Theorem 1.4]{LS23}, $$\vol(K_{X'}+B')\geq\frac{1}{462}>\frac{1}{825}$$ and we are done. Thus we may assume that $B'=0$. By Theorem \ref{thm: 825}, we have $\vol(K_{X'})\geq\frac{1}{825}$, and $\vol(K_{X'})=\frac{1}{825}$ if and only if $X'=V$ where $V$ is the surface as in Example \ref{ex: lc surface smallest volume}.
\end{proof}

Recall that $\KK^2:=\{K_X^2\mid X\in \sS\}$ and $\KK^2_\klt :=\{K_X^2\mid X\in \sS_\klt\}$. We introduce one more class of surface pairs 
\[
\sS(\{1\}):=\{(X, B)\mid (X,B)\text{ is a projective lc surface with $B$ reduced and $K_X+B$ ample}\},
\]
and the corresponding set of volumes
\[
\KK^2(\{1\}):=\{(K_X+B)^2\mid (X,B)\in \sS(\{1\})\},
\]
All of the three sets $\KK^2$, $\KK^2_\klt$, and $\KK^2(\{1\})$ are DCC sets (\cite[Theorem 8.2]{Ale94}, \cite[Theorem 1.3]{HMX14}), and so are their closure in $\RR$. In particular, each of them has a minimal accumulation point.
\begin{thm}\label{thm: smallest accu of surface vol general version}
The smallest accumulation points of $\KK^2_\klt$, $\KK^2$, and $\KK^2(\{1\})$ are all equal to $\frac{1}{825}$.
\end{thm}
\begin{proof}
   Let $v$ (resp.~$v'$ resp.~$v''$) be the smallest accumulation point of $\KK^2$ (resp.~$\KK^2_\klt$, resp.~$\KK^2(\{1\})$). Since $\KK^2_\klt\subset\KK^2\subset\KK^2(\{1\})$, it is clear that 
   \begin{equation}\label{eq: min acc}
       v'\geq v\geq v''.
   \end{equation} 

    By Corollary~\ref{cor: min acc}, $v=\frac{1}{825}$. By \cite[Theorem 1.1]{AL19b}, there exists a projective lc surface pair $(X,B)$ such that $B$ is a reduced divisor, $K_X+B$ is ample, and $(X,B)$ is not klt, such that $\vol(K_X+B)=v''$. By Theorem \ref{thm: general version lc pair vol min}, $v''\geq\frac{1}{825}$. Thus $v''=\frac{1}{825}$ by \eqref{eq: min acc} and the equality $v=\frac{1}{825}$.
    
    Since $v=\frac{1}{825}$ is an accumulation point of the DCC set $\KK^2$, there exists a sequence of projective lc surfaces $X_i\in \sS$ such that $K_{X_i}^2$ is strictly increasing, and 
    \[
\lim_{i\rightarrow+\infty}K_{X_i}^2=\frac{1}{825}.
    \]
     Since $K_{X_i}^2<\frac{1}{825}$ for each $i$, by Theorem \ref{thm: 825}, $X_i$ is klt for each $i$. Thus $\frac{1}{825}$ is an accumulation point of $\KK^2_\klt$, so $v'\leq\frac{1}{825}$. It follow from \eqref{eq: min acc} and the equality $v=\frac{1}{825}$ that $v'=\frac{1}{825}$.
\end{proof}

Theorem \ref{thm: 825} and Corollary~\ref{cor: min acc} also have implications for the classification of stable surfaces with volume less than $\frac{1}{825}$.

\begin{cor}\label{cor: finite <1/825}
For any positive rational number $v<\frac{1}{825}$, the following statements hold:
\begin{enumerate}
    \item $\KK^2_{\leq v}:=\{u\in \KK^2 \mid u\leq v\}$ is a finite set.
    \item $\sS_{\leq v}:=\{X\in \sS\mid K_X^2\leq v\}$ is a subset of $\sS_{\klt}$, and forms a bounded family of stable surfaces.
\end{enumerate}  
\end{cor}
\begin{proof}
By Corollary \ref{cor: min acc}, $\KK^2_{\leq v}$ is a DCC set of positive rational numbers without accumulation points and is therefore finite. The inclusion $\sS_{\leq v}\subset \sS_\klt$  follows directly from the assumption $v<\frac{1}{825}$, Theorem~\ref{thm: 825}, and the decomposition \eqref{eq: decompose S}. The boundedness of $\sS_{\leq v}$ follows from (1) and the fact that surfaces in $\sS$ with a fixed volume are bounded (\cite{Ale94, HMX18}).
\end{proof}

In view of the uniqueness assertion of Theorem \ref{thm: 825}, we propose the following question:
\begin{ques}
For any rational number $0<v<\frac{1}{825}$, is there, up to isomorphism, \emph{at most} one klt stable surface $X$ such that $K_X$ is ample and $K_X^2=v$? Note that any such surface must be rational (cf. \cite[Theorem 1.1]{Liu25}).
\end{ques}






\end{document}